\documentclass[12pt]{article}
\usepackage[english]{babel}
\usepackage{amssymb, amsfonts, amsthm, amsmath}
\usepackage{graphicx}

\newtheorem{theorem}{\bf Theorem}
\newtheorem{remark}{\bf Remark}

\begin{document}

\title{\textbf{\large{On Fixed Points of Nonlinear Monotone and Strongly Concave Operators Acting in Normal Cones}}
\author{Khachatur A. Khachatryan}}
\date{}

\maketitle

\textbf{Abstract:} We introduce and study a new class of nonlinear monotone operators
acting in normal cones of real Banach spaces and possessing the property
of strong concavity. We establish
new constructive principles for the existence of nonzero fixed points for this class of operators.
Further, we prove that the corresponding iterative process converges to the fixed point
at geometric rate. We also establish the uniqueness of the fixed point in a sufficiently wide conical segment. These results are applied to Hammerstein-type and Urysohn-type nonlinear integral operators acting in non-reflexive
Banach spaces, as well as to the Cauchy problem for a nonlinear heat equation.



\maketitle

\section{Introduction}\label{sec1}

\subsection{Basic notions and problem history}\label{subsec1.1}

First, we present the necessary definitions and known facts
from the theory of nonlinear monotone operators acting in cones
of real Banach spaces (see \cite{kras1}-\cite{kras4}).

Let $E$ be a real Banach space
partially ordered by a cone $K.$
The cone $K$ is called \emph{regular}
if every monotonically nondecreasing sequence that is bounded from above
converges in $E.$

The cone $K$ is called \emph{minihedral}
if any pair of elements of $E$ has a supremum.
We shall call the cone $K$ \emph{strongly minihedral}
if, for every set $M \subset E$ bounded from above,
the least upper bound $\sup M$ exists.

The cone $K$ is called \emph{normal} if there exists a number $\delta>0$ such that
\[
\left|\left| e_1 + e_2\right|\right| \ge \delta
\quad \text{for all } e_1,e_2 \in K
\text{ with } \left|\left|e_1\right|\right|=\left|\left|e_2\right|\right|=1.
\]
It is easy to verify that every regular cone is normal.
Generally speaking, the converse statement is not true.
The cone of nonnegative functions in the space of continuous functions
$C[a,b]$ is an example of a normal but not regular cone.

It is well-known that cones of nonnegative functions in the spaces
$C$ and $L_p$ are normal.
However, the cone of nonnegative functions in the space
$C^1[a,b]$ of continuously differentiable functions on the interval $[a,b]$
$x(t),$ endowed with the norm
\[
\|x\| = \max_{t\in[a,b]} |x(t)| + \max_{t\in[a,b]} |x'(t)|,
\]
does not possess the property of normality.

We say that a norm in the space $E$ with cone $K$ is
\emph{semi-monotone} if there exists a number $N>0$ such that,
for any elements $x,y \in K$ from $x \le y,$ we have
$
\|x\| \le N \|y\|.
$

We say that the norm is \emph{monotone} if from the inequalities
$\theta \le x \le y$ ($\theta$ is the zero element of the space $E$)
it follows that
$
\|x\| \le \|y\|.
$

It is easy to verify that in spaces $C$ and $L_p$
the norm is monotone with respect to the semi-ordering
induced by the cone of nonnegative functions.

It is well-known that semi-monotonicity of the norm is equivalent to
the normality of the cone  (see ~\cite{kras1}).
Let $u_0, v_0 \in E$ and $u_0 \le v_0.$
The set $\langle u_0, v_0 \rangle$ of elements $x \in E$ satisfying inequalities
$u_0 \le x \le v_0$ is called a \emph{conical segment}.
We say that an operator $A \colon E \to E$ is \emph{monotone}
if from $x \le y$ it follows that $Ax \le Ay.$

We say that the operator $A$ maps the conical segment
$\langle u_0, v_0 \rangle$ \emph{into itself}
(or \emph{leaves it invariant})
if
$
Au_0 \ge u_0$
and
$Av_0 \le v_0.
$

We say that the operator $A$ is \emph{critical}
if $A\theta = \theta$ (see~\cite{khach2}).

The following results have been established (see~\cite{kras1}, \cite{bakh3}) for operators that preserve the invariance of the conical segment $\langle u_0, v_0 \rangle.$

\medskip
\textbf{Theorem A} {\rm (see \cite{kras1}).}
\emph{Let $A$ be a monotone and continuous operator on the conical segment
$\langle u_0, v_0 \rangle.$
If the operator $A$ maps
the conical segment $\langle u_0, v_0 \rangle$ into itself,
and the cone $K$ is regular, then
the operator $A$ has at least one
fixed point $x^{*}$.}

\medskip

\textbf{Theorem B} {\rm (Birkhoff--Tarski, \cite{kras1}).}
\emph{Let the operator $A$ be monotone on the conical segment
$\langle u_0, v_0 \rangle$ and let it
map $\langle u_0, v_0 \rangle$ into itself.
If the cone $K$
be strongly minihedral, then there exists
 at least one fixed point $x^{*}\in \langle u_0, v_0 \rangle$ of the operator $A.$}
\medskip

\textbf{Theorem C}
{\rm (a consequence of Schauder's theorem, \cite{kras1}).}
\emph{Let a monotone, completely continuous operator $A$
map the conical segment
$\langle u_0, v_0 \rangle$ into itself. If the cone $K$ is normal, then
the operator $A$ has at least one
fixed point.}

\textbf{Theorem D}
{\rm (I.\,A.~Bakhtin — M.\,A.~Krasnosel'skii, \cite{kras1}, \cite{bakh3}).}
\emph{Let the operator $A$ be monotone and continuous on the conical segment
$\langle u_0, v_0 \rangle,$ and let it
satisfy the condition
\begin{equation}
A(x+y) \ge Ax + \alpha(\|y\|)\, z_0,
\qquad
(x,\, x+y \in \langle u_0, v_0 \rangle,\ y \in K),
\tag{$\ast$}
\end{equation}
where $\alpha(t)$ is a nondecreasing positive continuous function
 for $t>0,$
and $z_0$ is a fixed nonzero element of the cone $K.$
Suppose that the operator $A$ maps
$\langle u_0, v_0 \rangle$ into itself.
Then
there exists at least one
fixed point of the operator $A$ on the conical segment $\langle u_0, v_0 \rangle.$}

\medskip

Further, the notions of
\emph{concavity} and $u_0$--\emph{concavity}
of an operator $A$ on the cone $K$ were introduced in the works~\cite{kras1}, \cite{bakh3} as follows.

We say that a monotone operator $A$ is
\emph{concave on $K$} if it is positive and if
for every nonzero element $x \in K$
there exist positive numbers $\alpha$ and $\beta$
such that
\[
\alpha u_0 \le Ax \le \beta u_0, \ \text{where}\ \   u_0 \text{ is a fixed nonzero element of the cone } K,
\]
and for any $x \in K,$ satisfying $x \ge \gamma u_0$
$(\gamma > 0),$ the relation
\begin{equation}\label{Khachatryan1}
A(tx) \ge t\,Ax, \qquad A(tx) \ne t\,Ax,
\quad (0 < t < 1)
\end{equation}
holds.

We say that a monotone concave operator $A$ is
\emph{$u_0$-concave} if for any $x \in K$
$(x \ge \gamma u_0,\ \gamma > 0)$
the following condition, which is stronger than \eqref{Khachatryan1}, holds:
for every segment $[a,b] \subset (0,1),$
there exists a number $\eta = \eta(x,a,b) > 0$
such that
\begin{equation}\label{Khachatryan2}
A(tx) \ge (1+\eta)\, t\, Ax .
\end{equation}

The following results have been established for concave and $u_0$-concave operators (see \cite{kras1}, \cite{bakh3}).

\medskip

\textbf{Theorem E}
{\rm (I.\,A.~Bakhtin — M.\,A.~Krasnosel'skii, \cite{kras1}, \cite{bakh3}).}
\emph{Suppose that, for a concave operator $A,$ the equation $x = Ax$
has
a unique nonzero solution $x^{*}$ in the cone $K.$
Further, suppose that one of the following three conditions holds:
\begin{enumerate}
\item[(a)] the cone $K$ is regular;
\item[(b)] the operator $A$ is completely continuous;
\item[(c)] condition {\rm($\ast$)} from Theorem~D holds.
\end{enumerate}
Then for any nonzero initial approximation
$x_0 \in K,$ successive approximations
$
x_{n+1} = A x_n, \; n = 0,1,\ldots,
$
converge in norm to $x^{*}.$
}

\medskip

\textbf{Theorem F}
{\rm (I.\,A.~Bakhtin — M.\,A.~Krasnosel'skii, \cite{kras1}, \cite{bakh3}).}
\emph{Suppose that,  for a $u_0$-concave operator $A,$ equation $x = Ax$
has a nonzero solution $x^{*}$ in the cone $K.$
Then, for any nonzero initial approximation $x_0 \in K,$ successive approximations
$
x_{n+1} = A x_n, \; n = 0,1,\ldots,
$
converge to $x^{*}$ in the $u_0$-norm
(we refer to \cite{kras1} for the concept of $u_0$-norms).}

\medskip

In ~\cite{kras4} the above-mentioned results were applied to Hammerstein-type nonlinear integral
operators.

\subsection{Summary of obtained results.}\label{subsec1.2}

Nonlinear monotone operators possessing the property of criticality
arise in various areas of mathematical physics and mathematical
biology (see~\cite{bre5}-\cite{diek8} and the references therein).
In most cases, for such operators, it is necessary to construct
a second (nontrivial) fixed point in the cone $K.$ In various applied problems it is also of interest
to establish the uniqueness of this second fixed point
in a certain conical segment.

We introduce the notion of \emph{strong concavity}
of monotone operators as follows.
We say that a monotone operator $A$ is \emph{strongly concave}
in the subset $K_1 \subseteq K$ of the cone $K$
if there exists a continuous and monotonically increasing
concave mapping $\varphi \colon [0,1] \to [0,1]$ satisfying properties
$
\varphi(0)=0, $ $ \varphi(1)=1, \ \varphi'( +0 ) = +\infty,
$
such that
\begin{equation}\label{Khachatryan3}
A(\sigma u) \ge \varphi(\sigma)\, Au,
\qquad u \in K_1 \subseteq K,\ \sigma \in [0,1].
\end{equation}

We investigate the existence and uniqueness
of nontrivial fixed points for strongly concave operators.
Under certain natural conditions on strongly concave
operators, we prove constructive existence theorems
for nonzero fixed points in normal cones.
Moreover, we establish that
 the corresponding iterative process converges in the norm
with a rate determined by a decreasing geometric progression.
The geometric progression itself is determined by
the concave mapping~$\varphi.$

Further, we prove the uniqueness of the constructed
fixed point in a sufficiently wide conical segment.

We emphasize that the novelty of our assumptions and results are:
\begin{enumerate}
\item[I)] we require strong concavity of the operator under consideration;
\item[II)] in some cases, we do not assume that the operator is continuous;
\item[III)] we find the solution of the nonlinear operator
equation $Ax = x$ in normal cones, in a constructive way;
\item[IV)] we obtain the rate of convergence of the corresponding
iterative process;
\item[V)] we obtain a uniqueness theorem for the constructed
fixed point in a sufficiently wide conical segment,
without requiring the normality of the cone
or the continuity of the corresponding operator.
\end{enumerate}

We conclude our paper as follows.
Section~\ref{sec2} is devoted to existence theorems for fixed points for strongly concave monotone operators.
In section~\ref{sec3}  we provide some consequences and generalizations of those theorems.
Section~\ref{sec4}  we discuss uniqueness of fixed points or its absence.
Finally, the last section~\ref{sec5}  the obtained results we apply to Hammerstein-type  and Urysohn-type nonlinear integral equations, as well as to the Cauchy problem for the nonlinear heat equation.

\section{Existence theorems for fixed points}\label{sec2}

In this section, we state three existence theorems which together constitute the main result of our work in terms of existence.

\medskip

\begin{theorem}\label{thm1}
Let $A \colon K \to K$ be a monotone operator,
and let the cone $K$ be normal.
Assume that there exist numbers
$\sigma_0 \in (0,1),$ $n_0 \in \mathbb{N},$
and an element $v_0 \in K,$ $v_0 \neq \theta$
such that
\begin{equation}\label{Khachatryan4}
\sigma_0 A^{\,n_0-1} v_0
\;\le\;
A^{\,n_0} v_0
\;\le\;
A^{\,n_0-1} v_0 .
\end{equation}

If the operator $A$ is strongly concave
on $K_1 := \langle \theta, A^{\,n_0-1} v_0 \rangle \subset K,$
then the following assertions hold:
\begin{enumerate}
\item[1)]
there exists an element $x^{*} \in K$, $x^{*} \neq \theta$
such that
$
Ax^{*} = x^{*};
$
\item[2)]
there exist numbers $C>0$ and $k \in (0,1)$ such that
\[
\|A^{n} v_0 - x^{*}\| \le C k^{n},
\qquad n = n_0, n_0+1, \ldots .
\]
\end{enumerate}
\end{theorem}

\medskip

\begin{proof}
Consider the following iterations:
\begin{equation}\label{Khachatryan5}
x_{n+1}=Ax_n,\qquad x_0=A^{\,n_0-1}v_0,\qquad n=0,1,\ldots
\end{equation}

Using inequality $A^{n_0}v_0\le A^{n_0-1}v_0$ (see condition \eqref{Khachatryan4})
and the monotonicity of the operator $A,$ it is easy to prove by induction on $n$ that
\begin{equation}\label{Khachatryan6}
x_{n+1}\le x_n,\qquad n=0,1,\ldots
\end{equation}
Rewrite inequality \eqref{Khachatryan4} in terms of $x_0$ and $x_1$:
\begin{equation}\label{Khachatryan7}
\sigma_0 x_0 \le x_1 \le x_0.
\end{equation}
Taking into account \eqref{Khachatryan5}, as well as the monotonicity and strong concavity of the operator $A,$
from \eqref{Khachatryan7} we obtain
\begin{equation}\label{Khachatryan8}
\varphi(\sigma_0)\,x_1 \le x_2 \le x_1.
\end{equation}
Since $\varphi(\sigma_0)\in(0,1),$ using \eqref{Khachatryan5} again, as well as
the monotonicity and strong concavity of the operator $A,$ from \eqref{Khachatryan8} we obtain
\[
\varphi(\varphi(\sigma_0))\,x_2 \le A\bigl(\varphi(\sigma_0)x_1\bigr)
\le x_3 \le x_2.
\]
Continuing this process, at the $n$th step we obtain
\begin{equation}\label{Khachatryan9}
\underbrace{\varphi\left(\varphi\cdots \varphi\left(\sigma_0\right)\right)}_{n}\,x_n
\le x_{n+1} \le x_n.
\end{equation}
We now consider the following characteristic equation:
\begin{equation}\label{Khachatryan10}
\varphi(\tau)=\frac{\tau}{\sigma_0}.
\end{equation}
Let us verify that equation \eqref{Khachatryan10} has unique solution
$\tau^*\in(0,\sigma_0)$.
Indeed, consider the following function in interval $(0,1]:$
\[
\chi(\tau):=\frac{\varphi(\tau)}{\tau}-\frac{1}{\sigma_0}.
\]
Taking into account the properties of the mapping $\varphi$
(see Subsection~\ref{subsec1.2} for the definition of strong concavity of the operator),
we claim the following: $\chi \in C(0,1],$ $ \chi(\tau)$ is monotonically decreasing on $(0,1],$
$ \chi(1)=1-\frac{1}{\sigma_0}<0,$
$
\chi(+0):=\lim\limits_{\tau\to+0}\chi(\tau)
=\varphi'(+0)-\frac{1}{\sigma_0}=+\infty .
$
Consequently, there exists unique point
$\tau^{*}\in(0,1)$ such that
$
\chi(\tau^{*})=0.
$
Since the function $\dfrac{\varphi(\tau)}{\tau}$ is monotonically decreasing on $(0,1],$
we conclude that in fact $\tau^{*}\in(0,\sigma_0).$

Now, let us prove that the sequence of elements
$\{x_n\}_{n=0}^{\infty}$ satisfies the following
lower bound:
\begin{equation}\label{Khachatryan11}
x_n \ge \tau^{*} A^{\,n_0-1} v_0,
\qquad n=0,1,\ldots
\end{equation}

For $n=0,$ inequality~\eqref{Khachatryan11} follows immediately
from the definition of the initial approximation in iterations~\eqref{Khachatryan5}.
Assume that~\eqref{Khachatryan11} holds
for some natural number $n$.
Then, taking into account condition~\eqref{Khachatryan4}, as well as
the equality $\sigma_0\,\varphi(\tau^{*})=\tau^{*},$
the monotonicity and strong concavity of the operator $A,$
from~\eqref{Khachatryan5} we obtain
\[
x_{n+1}=Ax_n \ge A\!\left(\tau^{*}A^{\,n_0-1}v_0\right)
\ge \varphi(\tau^{*})A^{\,n_0}v_0
\ge \sigma_0\varphi(\tau^{*})A^{\,n_0-1}v_0
= \tau^{*}A^{\,n_0-1}v_0 .
\]

Now using~\eqref{Khachatryan6}, from~\eqref{Khachatryan9} we obtain
\begin{equation}\label{Khachatryan12}
\theta \le x_n-x_{n+1}
\le \Bigl(1-\underbrace{\varphi(\varphi\cdots\varphi(\sigma_0))}_{n}\Bigr)
A^{\,n_0-1}v_0,
\qquad n=1,2,\ldots
\end{equation}
where $\theta$ is the zero element of the cone $K$.

Since the cone $K$ is normal, it follows from \eqref{Khachatryan12}
that there exists a number $N_0>0$ such that
\begin{equation}\label{Khachatryan13}
\|x_n-x_{n+1}\|
\le
N_0\Bigl(1-\underbrace{\varphi(\varphi\cdots\varphi(\sigma_0))}_{n}\Bigr)
\bigl\|A^{\,n_0-1}v_0\bigr\|,
\qquad n=1,2,\ldots
\end{equation}

Now, using estimate (3.6) from paper~\cite{khach9}, from \eqref{Khachatryan13} we obtain
\begin{equation}\label{Khachatryan14}
\|x_n-x_{n+1}\|
\le
N_0(1-\sigma_0)\,\bigl\|A^{\,n_0-1}v_0\bigr\|\,k^{\,n},
\qquad n=1,2,\ldots,
\end{equation}
where
\begin{equation}\label{Khachatryan15}
k:=\frac{1-\varphi(\sigma_0)}{1-\sigma_0}\in(0,1).
\end{equation}

Writing inequality \eqref{Khachatryan14} for the indices $n+1,n+2,\ldots,n+m$
and using the triangle inequality, we obtain
\[
\|x_n-x_{n+m+1}\|
\le
\|x_n-x_{n+1}\|+\|x_{n+1}-x_{n+2}\|+\cdots+\|x_{n+m}-x_{n+m+1}\|\leq
\]
\[
\le
N_0(1-\sigma_0)\,\bigl\|A^{\,n_0-1}v_0\bigr\|\,
k^{\,n}\,(1+k+\cdots+k^{\,m})
\le
\frac{N_0(1-\sigma_0)\,\|A^{\,n_0-1}v_0\|}{1-k}\,k^{\,n},
\]
\[
\ n,m=1,2,\ldots.
\]

Thus, based on \eqref{Khachatryan15}, we obtain
\begin{equation}\label{Khachatryan16}
\|x_n-x_{n+m+1}\|
\le
\frac{N_0(1-\sigma_0)^2\,\|A^{\,n_0-1}v_0\|}{\varphi(\sigma_0)-\sigma_0}\,
k^{\,n},
\qquad n,m=1,2,\ldots.
\end{equation}

From \eqref{Khachatryan16}, in view of \eqref{Khachatryan15}, it follows that $\{x_n\}_{n=0}^{\infty}$ is a Cauchy sequence.
Since the space $E$ is complete, there exists $x^{*}\in E$ such that
$\|x_n-x^{*}\|\to 0$ as $n\to\infty$.
Since $K$ is closed, and $x_n\in K$, $n=0,1,\ldots,$
it follows that in fact $x^{*}\in K.$

Again, using that the cone $K$ is closed and taking into account
inequality~\eqref{Khachatryan11}, we obtain
\begin{equation}\label{Khachatryan17}
x^{*} \ge \tau^{*} A^{\,n_0-1} v_0.
\end{equation}

Consequently, $x^{*}\neq\theta.$

Letting the index $m$ tend to infinity in~\eqref{Khachatryan16}, and using
the continuity of the norm, we obtain
\begin{equation}\label{Khachatryan18}
\|x_n-x^{*}\|
\le
\frac{N_0(1-\sigma_0)^2\,\|A^{\,n_0-1}v_0\|}
{\varphi(\sigma_0)-\sigma_0},
\qquad n=1,2,\ldots.
\end{equation}

Thus, in order to complete the proof of the theorem,
we now verify that
$
Ax^{*}=x^{*}.
$

First, note that from~\eqref{Khachatryan6} it follows that
$
x^{*}\le x_n,$ $ n=0,1,\ldots,
$
whence, by the monotonicity of the operator $A,$ we obtain
\begin{equation}\label{Khachatryan19}
Ax^{*}\le Ax_n=x_{n+1}.
\end{equation}

From~\eqref{Khachatryan19}, taking into account that the cone $K$ is closed, it immediately follows that
\begin{equation}\label{Khachatryan20}
Ax^{*}\le x^{*}.
\end{equation}

Now, taking into account~\eqref{Khachatryan5}, \eqref{Khachatryan17} and~\eqref{Khachatryan20},
by monotonicity and strong concavity of the operator $A,$ we obtain
\begin{equation}\label{Khachatryan21}
\varphi(\tau^{*})\,x_1 \le Ax^{*} \le x^{*}.
\end{equation}

Again, using monotonicity and strong concavity of the operator $A,$
from~\eqref{Khachatryan21} we obtain
\[
\varphi\!\bigl(\varphi(\tau^{*})\bigr)\,x_2
\le
A(Ax^{*})
\le
Ax^{*}.
\]

By continuing this process for arbitrary $n\in\mathbb{N},$
we obtain the following chain of inequalities:
\[
\underbrace{\varphi(\varphi\cdots\varphi(\tau^{*}))}_{n}\,x_n
\le A^{n}x^{*}\le A^{n-1}x^{*}\le \cdots \le Ax^{*},
\]
whence, taking into account estimate (3.6) (see~\cite{khach9}), we obtain
\begin{equation}\label{Khachatryan22}
\Bigl(1-(1-\tau^{*})k_{*}^{\,n}\Bigr)x_n \le Ax^{*},
\qquad n=1,2,\ldots,
\end{equation}
where
\begin{equation}\label{Khachatryan23}
k_{*}:=\frac{1-\varphi(\tau^{*})}{1-\tau^{*}}\in(0,1).
\end{equation}

Letting the index $n$ tend to infinity in~\eqref{Khachatryan22} and taking into account
\eqref{Khachatryan23} and that the cone $K$ is closed, we obtain the inequality
\begin{equation}\label{Khachatryan24}
x^{*}\le Ax^{*}.
\end{equation}

From \eqref{Khachatryan20} and \eqref{Khachatryan24} it follows that
$
Ax^{*}=x^{*}.
$
The theorem is thus proved.
\end{proof}
\medskip
\begin{remark}
Below we present an example of a critical operator
$A:K\to K$ that is discontinuous yet satisfies conditions of Theorem~\ref{thm1} and as a consequence possesses
a nonzero fixed point. Let $E=\ell_\infty,$ and let
\[
K=\{\text{x}=(x_1,x_2,\ldots)\in \ell_\infty:\; x_i\ge0,\ i\in\mathbb{N}\}.
\]
Define the operator $A$ as follows:
$
A\text{x}=\text{y},\ \text{x}\in K,
$
where $\text{y}=(y_1,y_2,\ldots)\in K$,
$
y_n=\min\left\{n^{\frac14}\sqrt{x_n},\,1\right\}, \ n\in\mathbb{N}.
$
Obviously, $A:K\to K,$ and $A\theta=\theta,$ where
$\theta=(0,0,\ldots),$
and, for all $\emph{x,y}\in K$ from $\theta\le \text{x}\le \emph{y,}$ it follows that
$A\text{x}\le A\text{y}$.

Note that the operator $A$ is discontinuous
at the point $\theta\in K.$
Indeed, consider the following sequence
$
\text{x}^{(k)}=(x^{(k)}_1,x^{(k)}_2,\ldots)\in K,
$
where
\[
x_n^{(k)}=
\begin{cases}
\dfrac{1}{\sqrt{k}}, & n=k,\\[6pt]
0, & n\ne k,
\end{cases}
\]
then clearly
$
\left|\left|\text{x}^{(k)}\right|\right|_{\ell_\infty}=\frac1{\sqrt{k}}\to0,
\ k\to\infty.
$
However
\[
\left|\left|A\text{x}^{(k)}\right|\right|_{\ell_\infty}
=\sup\limits_{n\in\mathbb{N}}
\min\left\{n^{\frac14}\sqrt{x_n^{(k)}},1\right\}
=1.
\]

Choose $v_0,$ $\sigma_0$ and $n_0$ respectively as
$
v_0=(2,2,\ldots)\in K,
\
\sigma_0=\dfrac13, \ n_0=1.
$
Let us verify condition \eqref{Khachatryan4}.
Indeed, from the definition of the operator $A$ it immediately follows that $Av_0$'s $n$th coordinate satisfies
$$
(Av_0)_n=\min\{n^{\frac14}\sqrt{2},1\}=1\le2=(v_0)_n,
$$
$$
(Av_0)_n=1\ge\frac13\,(v_0)_n,\qquad n\in\mathbb{N}.
$$

Finally, let us verify the strong concavity condition.
Taking the mapping $\varphi(\sigma)$ to be
$
\varphi(\sigma)=\sqrt{\sigma},
$
from the definition of the operator $A,$ we obtain
$$
(A(\sigma u))_n
=\min\left\{n^{\frac14}\sqrt{\sigma u_n},\,1\right\}
=\min\left\{\sqrt{\sigma}\,n^{\frac14}\sqrt{u_n},\,1\right\}\ge
$$
$$
\ge
\sqrt{\sigma}\,
\min\left\{n^{\frac14}\sqrt{u_n},\,1\right\}
=
\sqrt{\sigma}\,(Au)_n,
\qquad n\in\mathbb{N},
$$
where $u=(u_1,u_2,\ldots)\in K$.

The validity of the inequality
\[
\min\left\{\sqrt{\sigma}\,n^{\frac14}\sqrt{u_n},\,1\right\}
\ge
\sqrt{\sigma}\,
\min\left\{n^{\frac14}\sqrt{u_n},\,1\right\},
\qquad n\in\mathbb{N},\; \sigma\in[0,1],
\]
can be verified by considering the cases
$
n^{\frac14}\sqrt{u_n}\le1
\quad\text{and}\quad
n^{\frac14}\sqrt{u_n}>1
$ separately.

Using the fact that in this case
the cone $K$ is normal (since the norm is monotone)
and applying Theorem~\ref{thm1}, we conclude that the operator $A$
has a nonzero fixed point
$
\text{x}^*\in K.
$
One can directly check that
$
\text{x}^*=(1,1,\ldots).
$
\end{remark}

\begin{remark}
It should be noted that in the course of the proof of Theorem~\ref{thm1} we also obtained
the following two-sided inequality for the fixed point $x^{*}$:
\begin{equation}\label{Khachatryan25}
\tau^{*}A^{\,n_0-1}v_0 \le x^{*}\le A^{\,n_0-1}v_0 .
\end{equation}
\end{remark}

Let $r_0>1$ be an arbitrary number.
Consider the following characteristic equation on $[1,+\infty)$:
\begin{equation}\label{Khachatryan26}
\delta\,\varphi\!\left(\frac{1}{\delta}\right)=r_0.
\end{equation}

By repeating the same arguments used in the study of the characteristic equation~\eqref{Khachatryan10}, one can prove that equation~\eqref{Khachatryan26} has a unique solution $\delta>r_0$.

Taking into account~\eqref{Khachatryan26}, we can prove the following theorem in an analogous way to the proof of
Theorem~\ref{thm1}.

\medskip

\begin{theorem}\label{thm2}
Let $A\colon K\to K$ be a monotone operator, and let the cone $K$ be
normal. Suppose that there exist numbers $r_0>1$, $n_0\in\mathbb{N},$
and an element $v_0\in K$, $v_0\neq\theta$, such that
\begin{equation}\label{Khachatryan27}
A^{\,n_0-1}v_0 \le A^{\,n_0}v_0 \le r_0\,A^{\,n_0-1}v_0 .
\end{equation}

Then, if the operator $A$ is strongly concave on
$
K_1:=\langle A^{\,n_0-1}v_0,\ \delta A^{\,n_0-1}v_0\rangle \subset K,
$ where $\delta>r_0$ is the unique solution of the characteristic
equation~\eqref{Khachatryan26}, then
\begin{enumerate}
\item[1)] there exists $x^{*}\in K$, $x^{*}\neq\theta$, such that
$
Ax^{*}=x^{*};
$
\item[2)] there exist numbers $C_0>0$ and $k_0\in(0,1)$ such that
\[
\|A^{n}v_0-x^{*}\|\le C_0\,k_0^{\,n},
\qquad n=n_0,n_0+1,\ldots.
\]
\end{enumerate}\qed
\end{theorem}

\medskip

\begin{remark}
Note that in Theorems~\ref{thm1} and~\ref{thm2} the corresponding conditions~\eqref{Khachatryan4} and~\eqref{Khachatryan27}
ensure the monotonicity of the successive approximations~\eqref{Khachatryan5}.
In the case of condition~\eqref{Khachatryan4} these approximations are monotonically decreasing,
whereas in the case of condition~\eqref{Khachatryan27} they are monotonically increasing.
The strong concavity property of the operator $A,$ together with the monotonicity of the successive approximations~\eqref{Khachatryan5}
makes it possible to prove the existence of a fixed point theorem without requiring continuity of the operator $A$.
\end{remark}

An interesting question arises: if the operator
$A\colon K\to K$ is continuous, monotone and strongly
concave, then can condition~\eqref{Khachatryan4} (or condition~\eqref{Khachatryan27})
be weakened while still being able to prove a result analogous to Theorems~\ref{thm1} and~\ref{thm2}?

The methods developed for proving Theorems~\ref{thm1} and~\ref{thm2} allow us
to answer this question in a certain sense that will be clarified below. In particular, we will request that the operator is continuous. Specifically, the following theorem holds.

\medskip

\begin{theorem}\label{thm3}
Let $A\colon K\to K$ be a monotone and continuous operator,
and let the cone $K$ be normal.
Suppose that there exist numbers
$
r_1\in(0,1), \ r_2\in(1,+\infty), \ n_0\in\mathbb{N},
$
and an element $v_0\in K$, $v_0\neq\theta$ such that
\begin{equation}\label{Khachatryan28}
r_1 A^{\,n_0-1}v_0 \le A^{\,n_0}v_0 \le r_2 A^{\,n_0-1}v_0 .
\end{equation}

If the operator $A$ is strongly concave on
$
K_1:=\langle \tau_1 A^{\,n_0-1}v_0,\ \delta_1 A^{\,n_0-1}v_0\rangle
\subset K,
$
where $\tau_1\in(0,r_1)$, $\delta_1\in(r_2,+\infty)$ are
the unique solutions of the characteristic equations
$
\varphi(\tau)=\frac{\tau}{r_1},
\
\delta\,\varphi\!\left(\frac{1}{\delta}\right)=r_2,
$
respectively, then
\begin{enumerate}
\item[1)] there exists $x^{*}\in K$, $x^{*}\neq\theta$, such that
$
Ax^{*}=x^{*};
$
\item[2)] there exist numbers $C_{*}>0$ and $k_{*}\in(0,1)$ such that
\[
\|A^{n}v_0-x^{*}\|\le C_{*}k_{*}^{\,n},
\qquad n=n_0,n_0+1,\ldots.
\]
\end{enumerate}
\end{theorem}
\begin{proof}
Consider the successive approximations~\eqref{Khachatryan5}.
Arguing in a way analogous to the proof of inequalities~\eqref{Khachatryan9}
and~\eqref{Khachatryan11}, and using induction with respect to $n$ as well as using
condition~\eqref{Khachatryan28}, one can prove that the sequence
$\{x_n\}_{n=0}^{\infty}$ satisfies the following inequalities:
\begin{equation}\label{Khachatryan29}
\underbrace{\varphi(\varphi\cdots\varphi(r_1))}_{n}\,x_n
\le x_{n+1}
\le
\frac{1}{\underbrace{\varphi(\varphi\cdots\varphi(1/r_2))}_{n}}\,x_n,
\qquad n=1,2,\ldots,
\end{equation}
\begin{equation}\label{Khachatryan30}
x_n \ge \tau_1 A^{\,n_0-1}v_0,
\qquad n=0,1,\ldots.
\end{equation}

Let us now prove that
\begin{equation}\label{Khachatryan31}
x_n \le \delta_1 A^{\,n_0-1}v_0,
\qquad n=0,1,\ldots.
\end{equation}

In the case $n=0$, estimate~\eqref{Khachatryan31} follows immediately
from the definition of the initial approximation in the iterations~\eqref{Khachatryan5}
taking into account the inequalities $\delta_1>r_2>1$.
Assume that~\eqref{Khachatryan31} holds for some $n\in\mathbb{N}$.
Then, taking into account the monotonicity and strong concavity
of the operator $A$, as well as~\eqref{Khachatryan28} and the equality
\[
\delta_1\varphi\!\left(\frac{1}{\delta_1}\right)=r_2,
\]
from~\eqref{Khachatryan5} we obtain
\[
x_{n+1}=Ax_n
\le A\!\left(\delta_1 A^{\,n_0-1}v_0\right)
\le
\frac{1}{\varphi(1/\delta_1)}\,A^{\,n_0}v_0
\le$$$$\le
\frac{r_2}{\varphi(1/\delta_1)}\,A^{\,n_0-1}v_0
=
\delta_1 A^{\,n_0-1}v_0 .
\]

From~\eqref{Khachatryan29} and~\eqref{Khachatryan31} it follows that
\begin{equation}\label{Khachatryan32}
\begin{array}{c}
\theta
\le \frac{1}{\underbrace{\varphi(\varphi\cdots\varphi(1/r_2))}_{n}} x_n-x_{n+1}
\le\\ \le
\Bigl(\frac{1}{\underbrace{\varphi(\varphi\cdots\varphi(1/r_2))}_{n}}-\underbrace{\varphi(\varphi\cdots\varphi(r_1))}_{n}\Bigr)
\,\delta_1 A^{\,n_0-1}v_0 .
\end{array}
\end{equation}

Taking into account the normality of the cone $K$ and inequality~\eqref{Khachatryan32},
one can assert that there exists a number $N_0>0$ such that
\begin{equation}\label{Khachatryan33}
\begin{array}{c}
\left\|
\frac{1}{\underbrace{\varphi(\varphi\cdots\varphi(1/r_2))}_{n}}
x_n-x_{n+1}
\right\|
\le\\ \le
N_0\,\delta_1\,\|A^{\,n_0-1}v_0\|
\left(
\frac{1}{\underbrace{\varphi(\varphi\cdots\varphi(1/r_2))}_{n}}
-
\underbrace{\varphi(\varphi\cdots\varphi(r_1))}_{n}
\right).
\end{array}
\end{equation}

From~\eqref{Khachatryan33}, taking into account estimate~(3.6) from paper~\cite{khach9},
as well as~\eqref{Khachatryan30}, \eqref{Khachatryan31} and the normality of the cone $K$,
it follows that
\[
\|x_n-x_{n+1}\|
\le
\left(
\frac{1}{\underbrace{\varphi(\varphi\cdots\varphi(1/r_2))}_{n}}
-1
\right)\|x_n\|
+
\left\|
\frac{1}{\underbrace{\varphi(\varphi\cdots\varphi(1/r_2))}_{n}}
x_n-x_{n+1}
\right\|\le
\]
\[
\le
N_1(r_2-1)\,\tilde k_2^{\,n}\,
\delta_1\|A^{\,n_0-1}v_0\|
+
\bigl(N_0\delta_1(r_2-1)\tilde k_2^{\,n}
+
N_0\delta_1(1-r_1)\tilde k_1^{\,n}\bigr)
\|A^{\,n_0-1}v_0\|\le
\]
\[
\le
\bigl(N_1(r_2-1)+N_0(r_2-r_1)\bigr)\,
\delta_1\,\|A^{\,n_0-1}v_0\|\,\tilde k^{\,n},
\qquad n=1,2,\ldots,
\]
where $N_1>0$
\[
\tilde k_1:=\frac{1-\varphi(1/r_2)}{1-1/r_2}\in(0,1),
\
\tilde k_2:=\frac{1-\varphi(r_1)}{1-r_1}\in(0,1),\
\tilde k:=\max\{\tilde k_1,\tilde k_2\}\in(0,1).
\]

Further, performing steps analogous to those in the proof of
Theorem~\ref{thm1} completes the proof.
\end{proof}

\section{Some consequences and a generalization of the existence theorems for fixed points}\label{sec3}

The following theorem is a is a consequence of Theorem~\ref{thm1}:

\medskip
\begin{theorem}\label{thm4}
Let $A_0\colon K\to K$ be a monotone and left-continuous operator,
and let $K$ be a proper cone.
Then, if there exist numbers $n_0\in\mathbb{N}$,
$\alpha,\gamma_0\in(0,1)$ and an element $v_0\in K$, $v_0\neq\theta$ such that
the operator $A_0$ satisfies the following two-sided inequality:
\begin{equation}\label{Khachatryan34}
A^{n_0-1}v_0-\gamma_0^{\,1-\alpha}
A\!\left(A^{n_0-1}v_0-u\right)
\;\ge\;
A_0u
\;\ge\;
A^{n_0-1}v_0-
A\!\left(A^{n_0-1}v_0-u\right)
\end{equation}
for all $u\in \langle \theta, A^{n_0-1}v_0 \rangle$,
where the operator $A$ satisfies the assumptions of Theorem~\ref{thm1} with
$\varphi(\sigma)=\sigma^{\alpha}$ and
$A^{n_0}v_0\neq A^{n_0-1}v_0$,
then there exists an element
$
\tilde{x}\in \langle \theta, A^{n_0-1}v_0 \rangle,\ \tilde{x}\neq\theta,\
\tilde{x}\neq A^{n_0-1}v_0
$
such that
$
A_0\tilde{x}=\tilde{x}.
$
\end{theorem}

\medskip
\begin{proof}
Since every proper cone is normal (see~\cite{kras1}),
and the function \linebreak $\varphi(\sigma)=\sigma^{\alpha}$, $\alpha\in(0,1),$
satisfies all properties required in the definition of strong
concavity of the operator $A,$ it follows from Theorem~\ref{thm1}, that there exists
$x^*\in K$, $x^*\neq\theta,$,
$x^*\le A^{n_0-1}v_0$ such that
$
Ax^*=x^*.
$

Now consider now the following successive approximations:
\begin{equation}\label{Khachatryan35}
x_{n+1}=A_0x_n,\qquad
x_0=A^{n_0-1}v_0-x^*,\qquad n=0,1,\ldots.
\end{equation}

By induction on $n,$ it is easy to verify that
$x_n\in K$, $n=0,1,\ldots$.

Let us now verify that
$
x_n \le x_{n+1}, \ n=0,1,\ldots.
$
Taking into account the right-hand side of inequality~\eqref{Khachatryan34}
and the equality
$Ax^*=x^*,$ from~\eqref{Khachatryan35} we obtain
\[
x_1=A_0x_0 \ge A^{n_0-1}v_0-
A\!\left(A^{n_0-1}v_0-x_0\right)
= A^{n_0-1}v_0-Ax^*
= A^{n_0-1}v_0-x^*=x_0.
\]

Suppose now that
$
x_{n-1}\le x_n
$
for some natural $n$.
Then, taking into account the monotonicity of the operator $A_0$
and the induction hypothesis, we obtain
\[
x_{n+1}=A_0x_n \ge A_0x_{n-1}=x_n.
\]

Let us now verify the following upper estimate:
\begin{equation}\label{Khachatryan36}
x_n \le A^{n_0-1}v_0-\gamma_0 x^*.
\qquad n=0,1,\ldots
\end{equation}

In the case $n=0$, this inequality follows immediately from inclusions
$\gamma_0\in(0,1)$ and $x^*\in K$.
Assume that inequality~\eqref{Khachatryan36} holds
for some $n\in\mathbb{N}$.
Then, using the monotonicity of the operator $A_0$
the strong concavity of the operator $A$, and
by the induction hypothesis, we have
\begin{align*}
x_{n+1}
&=A_0x_n
\le A_0\!\left(A^{n_0-1}v_0-\gamma_0 x^*\right)\le \\
&\le A^{n_0-1}v_0-
\gamma_0^{\,1-\alpha}
A\!\left(A^{n_0-1}v_0-A^{n_0-1}v_0+\gamma_0x^*\right)= \\
&=A^{n_0-1}v_0-
\gamma_0^{\,1-\alpha}A(\gamma_0x^*)
\le A^{n_0-1}v_0-\gamma_0 x^*.
\end{align*}

Thus, by virtue of~\eqref{Khachatryan36}, the monotonicity of
$\{x_n\}_{n=0}^{\infty}$ and the properness of the cone $K$,
we conclude that the sequence
$\{x_n\}_{n=0}^{\infty}$ converges in the norm to some element
$\tilde{x}\in K$.
Since the operator $A_0$ is left-continuous, it follows from the
monotonicity of the sequence $\{x_n\}_{n=0}^{\infty}$ that
$
A\tilde{x}=\tilde{x},
$
and
\begin{equation}\label{Khachatryan37}
A^{n_0-1}v_0-\gamma_0 x^* \ge \tilde{x} \ge A^{n_0-1}v_0-x^*.
\end{equation}

Let us now note that $\tilde{x}\neq \theta$ and
$\tilde{x}\neq A^{n_0-1}v_0$.
Indeed, since $Ax^*=x^*,$ $x^*\le A^{n_0-1}v_0,$
$A^{n_0-1}v_0\ge A^{n_0}v_0$ and
$A^{n_0}v_0\neq A^{n_0-1}v_0$,
it follows from~\eqref{Khachatryan37} that $\tilde{x}\neq \theta$.
On the other hand, since $x^*\in K$, $x^*\neq \theta$,
from~\eqref{Khachatryan37} and the inclusion $\gamma_0\in(0,1),$
we also obtain that
$\tilde{x}\neq A^{n_0-1}v_0$.
Thus, the theorem is proved.
\end{proof}

\medskip
\begin{remark}
Generally speaking, the uniqueness
of a nontrivial fixed point of the operator $A_0$
does not hold: for example, for critical operators $A,$
it follows from~\eqref{Khachatryan34} that
$
A_0\!\left(A^{n_0-1}v_0\right)=A^{n_0-1}v_0.
$
\end{remark}
The following theorem also holds.

\medskip
\begin{theorem}\label{thm5}
Let the monotone and strongly concave operator
$A:K\to K$ have a fixed point $x^*\in K$, $x^*\neq \theta$,
and let the cone $K$ be normal.
Further, let $A_0:K\to K$ be a monotone and critical operator.
Then, if there exists a number $C_0>0$ such that
$
A_0u \le C_0 Au,\ u\in K
$
and
$
A_0(tu)\ge tA_0u,\ t\in[0,1],\ u\in \langle x^*, r_*x^* \rangle,
$
(where $r_*>C_0+1$ is the unique solution of the characteristic
equation
$
r\varphi\!\left(\frac{1}{r}\right)=C_0+1
$),
then
\begin{enumerate}
\item[1)] the operator $A+A_0$ has a nontrivial fixed point
$\tilde{x}$ in the cone $K$;

\item[2)] there exist numbers $C_\sharp>0$ and $k_\sharp\in(0,1)$ such that
$$
\|y_n-\tilde{x}\|\le C_\sharp k_\sharp^{\,n},  n=1,2,\dots,
$$
where the sequence $\{y_n\}_{n=0}^{\infty}\subset K$ is defined
by the recurrence relations
\begin{equation}\label{Khachatryan38}
y_{n+1}=(A+A_0)y_n,\qquad y_0=x^*,\quad n=0,1,\dots;
\end{equation}

\item[3)] there exist numbers $C_*>0$ and $k_*\in(0,1)$ such that
$$
\|A^n\tilde{x}- x^*\|\le C_* k_*^{\,n},  n=1,2,\dots .
$$
\end{enumerate}
\end{theorem}

\begin{proof}
\medskip
First, let us verify that the operator $A+A_0$ is strongly concave
on the cone
$
K_1:=\langle x^*,\ r_* x^*\rangle \subset K .
$
For this purpose we define the mapping
\begin{equation}\label{Khachatryan39}
\Phi(\sigma):=\frac{C_0}{1+C_0}\,\sigma
+\frac{1}{1+C_0}\,\varphi(\sigma),
\qquad \sigma\in[0,1],
\end{equation}
where $\varphi(\sigma)$ is the concave mapping
corresponding to the operator $A$.

We show that
\begin{equation}\label{Khachatryan40}
(A+A_0)(\sigma u)\ge \Phi(\sigma)(A+A_0)u,
\qquad
u\in\langle x^*,r_* x^*\rangle,\ \sigma\in[0,1].
\end{equation}

Indeed, taking into account the strong concavity of the operator $A$
and the obvious inequality
$
\varphi(\sigma)\ge \sigma,\ \sigma\in[0,1],
$
and using \eqref{Khachatryan39}, we obtain
$$
(A+A_0)(\sigma u)-\Phi(\sigma)(A+A_0)u
=A(\sigma u)+A_0(\sigma u)
-\Phi(\sigma)\bigl(Au+A_0u\bigr)\ge
$$
$$
\ge \varphi(\sigma)Au + \sigma A_0u
- \Phi(\sigma)Au - \Phi(\sigma)A_0u
=$$$$=
\frac{\varphi(\sigma)-\sigma}{1+C_0}
\bigl(C_0Au - A_0u\bigr)
\;\ge\; \theta,
u \in \langle x^*, r_* x^* \rangle,
\sigma \in [0,1].
$$
We used
$A_0u \le C_0Au$, $u \in K$ to obtain the last inequality.

Now consider the successive approximations \eqref{Khachatryan38}.
By induction on $n$ it is easy to verify that
\begin{equation}\label{Khachatryan41}
y_{n+1} \ge y_n, \qquad n = 0,1,\dots .
\end{equation}

Let us prove that
\begin{equation}\label{Khachatryan42}
y_n \le r_* x^*, \qquad n = 0,1,\dots .
\end{equation}

For $n=0$ inequality \eqref{Khachatryan42} is obvious.
Suppose that \eqref{Khachatryan42} holds
for some natural $n$.
Then, using the condition $A_0u \le C_0Au$, $u \in K$ and
the strong concavity of the operator $A,$ as well as the monotonicity
of the operator $A+A_0,$ from \eqref{Khachatryan38} we obtain
$$
y_{n+1}
\le (A + A_0)(r_* x^*)
\le (C_0 + 1) A(r_* x^*)
\le \frac{C_0 + 1}{\varphi(1/r_*)} A x^*
= r_* x^* .
$$

Thus, from \eqref{Khachatryan41}, \eqref{Khachatryan42} we obtain the following full chain of inequalities:
\begin{equation*}
y_0 \le y_1  \le r_*y_0.
\end{equation*}
Let us now note that the mapping $\Phi$,
defined by formula \eqref{Khachatryan39},
possesses the same properties
as the mapping $\varphi$
(see the definition of strong concavity for the properties of $\varphi$).

Further, by repeating the arguments
analogous to the proof of Theorem~\ref{thm1},
and using the monotonicity and strong concavity
of the operator $A + A_0$, we arrive at statements 1) and 2)
of the formulated theorem.

In order to complete the proof, we are left to verify statement 3) of the theorem.
First note that from \eqref{Khachatryan41}, \eqref{Khachatryan42} it immediately follows that
\begin{equation}\label{Khachatryan43}
x^{*} \le y_n \le r_*\, x^{*}, \qquad n = 0,1,\dots .
\end{equation}

Taking into account the closedness of the cone $K$, from \eqref{Khachatryan43}
we conclude the following two-sided inequality:
\begin{equation}\label{Khachatryan44}
x^{*} \le \tilde{x} \le r_*\, x^{*}.
\end{equation}

Using the monotonicity and strong concavity of the operator $A,$
as well as the relation
$
A x^{*} = x^{*},
$
from \eqref{Khachatryan44} we obtain the estimates
\begin{equation}\label{Khachatryan45}
x^{*} \le A^{n} \tilde{x}
\le \frac{1}{\underbrace{\varphi\bigl(\varphi(\dots \varphi(1/r_*))\bigr)}_n}\, x^{*},
\qquad n = 1,2,\dots .
\end{equation}

Further,  using inequality (3.6) from paper~\cite{khach9}
and taking into account \eqref{Khachatryan45}, by virtue of the normality
of the cone $K,$ we conclude that there exists a constant $C_1 > 0$ such that
$$
\|A^{n} \tilde{x} - x^{*}\|
\le C_1 \|x^{*}\|
\left(
\frac{1}{\underbrace{\varphi\bigl(\varphi(\dots \varphi(1/r_*))\bigr)}_n} - 1
\right)\le C_1\, r_* \|x^{*}\|\, k_*^{\,n},
\ n = 1,2,\dots ,
$$
where
$
k_* :=
\frac{1 - \varphi(1/r_*)}{1 - 1/r_*}
\in (0,1).
$

Thus, the theorem is proved.
\end{proof}
\medskip
\begin{remark}
In the course of proving Theorem~\ref{thm5} we also obtained
the a priori two-sided estimate~\eqref{Khachatryan44}.
\end{remark}

The following theorem acts as a substantial generalization of Theorem~\ref{thm3}
in the case where the operator $A$ is strongly concave
on the whole cone $K$:

\medskip
\begin{theorem}\label{thm6}
Let $A\colon K \to K$ be a continuous, monotone and strongly
concave operator in a normal cone $K$.
Then, if there exist numbers
$r_1 \in (0,1)$, $r_2 > 1$, $n_0, m_0 \in \mathbb{N}$,
$n_0 \ge m_0$ and an element $v_0 \in K$, $v_0 \ne \theta$ such that
\begin{equation}\label{Khachatryan46}
r_1 A^{\,n_0-m_0} v_0
\le A^{\,n_0} v_0
\le r_2 A^{\,n_0-m_0} v_0,
\end{equation}
then there exist nonzero elements
$p_0, p_1, \dots, p_{m_0-1} \in K$ such that
$
A^{\,m_0} p_j = p_j,
\ j = 0,1,\dots,m_0-1.
$
Moreover, if additionally there exist numbers
$d_1 \in (0,1)$, $d_2 > 1$ and indices
$i_0, j_0 \in \{0,1,\dots,m_0-1\}$,
$i_0 \ne j_0$, such that $m_0$ and $|i_0-j_0|$ are relatively prime, and
\begin{equation}\label{Khachatryan47}
d_1 A^{\,m_0+i_0} v_0
\le A^{\,m_0+j_0} v_0
\le d_2 A^{\,m_0+i_0} v_0,
\end{equation}
then
$
p_0 = p_1 = \dots = p_{m_0-1} =: p^{*}
$
and
$
A p^{*} = p^{*}.
$
\end{theorem}
\begin{proof}
Define the successive approximations:
\begin{equation}\label{Khachatryan48}
x_{n+1} = A x_n,
\qquad
x_0 = A^{\,n_0-m_0} v_0,
\qquad
n = 0,1,\dots .
\end{equation}

Taking into account condition \eqref{Khachatryan46}, from \eqref{Khachatryan48} we obtain
\begin{equation}\label{Khachatryan49}
r_1 x_0 \le x_{m_0} \le r_2 x_0 .
\end{equation}

Using the monotonicity and strong concavity of the operator $A$
and performing steps analogous to those used in deriving
inequalities~\eqref{Khachatryan29}, as well as using  \eqref{Khachatryan48}, from \eqref{Khachatryan49}
we arrive at the following two-sided inequality:
\begin{align}\label{Khachatryan50}&
\underbrace{\varphi\bigl(\varphi(\dots \varphi(r_1))\bigr)}_{m-m_0}\, x_{m-m_0}
\le x_m
\le
\frac{1}{\underbrace{\varphi\bigl(\varphi(\dots \varphi(1/r_2))\bigr)}_{m-m_0}}\, x_{m-m_0},\\&\nonumber
 m = m_0, m_0+1,\dots.
\end{align}

Consider the following subsequences of the sequence
$\{x_n\}_{n=0}^{\infty}$:
\begin{align}\label{Khachatryan51}
& p_k^{\,0} := x_{m_0 k},
\,
p_k^{\,1} := x_{m_0 k + 1},
\,
 \dots
 ,\,
p_k^{\,m_0-1} := x_{m_0 k + m_0 - 1},
\\&\nonumber
\quad k = 1,2,\dots .
\end{align}

Then from (50) taking into account (51) we obtain the following inequalities:
\begin{align}\label{Khachatryan52}&
\underbrace{\varphi\bigl(\varphi(\dots \varphi(r_1))\bigr)}_{m_0(k-1)+j} p_{k-1}^{\,j}
\le p_k^{\,j}
\le
\frac{1}{\underbrace{\varphi\bigl(\varphi(\dots \varphi(1/r_2))\bigr)}_{m_0(k-1)+j}} p_{k-1}^{\,j},\\&
\nonumber
 k = 2,3,\dots, j=0,1,\dots m_0-1 .
\end{align}

Consider the following characteristic equations:
\begin{equation}\label{Khachatryan53}
\varphi(\tau) = \frac{\tau}{r_1},\ \  \tau \in (0,1),\ \
r\,\varphi\!\left(\frac{1}{r}\right) = r_2,
\ \ r \in (1,+\infty).
\end{equation}

Just as in the proofs of Theorems~\ref{thm1} and~\ref{thm3},  here we also have
$\tau^{*} \in (0,r_1)$ and $ r^{*} > r_2$
such that
$
\varphi(\tau^{*}) = \frac{\tau^{*}}{r_1},
\
 r^{*}\,\varphi\!\left(\frac{1}{ r^{*}}\right) = r_2.
$

On the other hand, from \eqref{Khachatryan51} and \eqref{Khachatryan48}
it immediately follows that:
\begin{align}\label{Khachatryan54}&
p_{k+1}^{\,0} = A p_k^{\,m_0-1}, \\&\nonumber
p_{k+1}^{\,1} = A p_{k+1}^{\,0}, \\&\nonumber \dots, \\&\nonumber
p_{k+1}^{\,m_0-1} = A p_{k+1}^{\,m_0-2},
\qquad k = 1,2,\dots .
\end{align}

By induction on $k,$ we prove the validity of the following set
of two-sided inequalities:
\begin{align}\label{Khachatryan55}&
r^{*} p_{1}^{\,0} \ge p_k^{\,0} \ge \tau^{*} p_{1}^{\,0}, \\&\nonumber
r^{*} p_{1}^{\,1} \ge p_k^{\,1} \ge \tau^{*} p_{1}^{\,1},\\&\nonumber
 \dots , \\&\nonumber r^{*} p_{1}^{\,m_0-1} \ge p_k^{\,m_0-1}
\ge \tau^{*} p_{1}^{\,m_0-1},  k = 1,2,\dots .
\end{align}

For $k = 1$ the above set of inequalities obviously holds. Indeed,
$ r^{*} > 1,$ $\tau^{*} \in (0,1),$ and
$p_1^{\,j} = x_{m_0 + j} \in K,$ $j = 0,1,\dots,m_0-1.$

Assume that inequalities \eqref{Khachatryan55} hold
for some natural $k$.
Then, taking into account the monotonicity and strong concavity
of the operator $A$, by virtue of
\eqref{Khachatryan54}, \eqref{Khachatryan52}, \eqref{Khachatryan50},
\eqref{Khachatryan51}, and using the induction hypothesis, we obtain
$$
p_{k+1}^{\,0}
= A p_k^{\,m_0-1}
\ge A\!\left(\tau^{*} p_{1}^{\,m_0-1}\right)
\ge \varphi(\tau^{*})\, A p_{1}^{\,m_0-1}
= \varphi(\tau^{*})\, p_2^{\,0}\ge $$$$\ge
\varphi(\tau^*) \underbrace{\varphi(\varphi\ldots\varphi(r_1))}_{m_0}p_1^0\ge r_1\varphi(\tau^*)p_1^0= r^{*} p_1^{\,0},
$$
$$
p_{k+1}^{\,1}
= A p_{k+1}^{\,0}
\ge A\!\left(\tau^{*} p_{1}^{\,0}\right)
\ge \varphi(\tau^{*})\, A p_{1}^{\,0}
=\varphi(\tau^*)Ax_{m_0}=\varphi(\tau^*)x_{m_0+1}=$$$$=\varphi(r^{*})\, p_1^{\,1}
\ge \tau^{*} p_1^{\,1},
$$
$$
\dots\dots\dots\dots\dots\dots\dots\dots\dots\dots\dots\dots\dots\dots\dots\dots\dots\dots\dots\dots\dots
$$
$$
p_{k+1}^{\,m_0-1}
= A p_{k+1}^{\,m_0-2}
\ge A\!\left(\tau^{*} p_{1}^{\,m_0-2}\right)
\ge \varphi(\tau^{*})\, A p_{1}^{\,m_0-2}
= \varphi(\tau^{*})x_{\,2m_0-1}=$$$$=\varphi(\tau^*)p_1^{m_0-1}
\ge \tau^{*} p_1^{\,m_0-1},
$$
$$
p_{k+1}^{\,0}
= A p_k^{\,m_0-1}
\le A\!\left(r^{*} p_{1}^{\,m_0-1}\right)
\le \frac{1}{\varphi\!\left(\frac{1}{r^{*}}\right)}\, A p_{1}^{\,m_0-1}
= \frac{1}{\varphi\!\left(\frac{1}{r^{*}}\right)}\, p_2^{\,0}\le $$
$$ \le  \frac{1}{\varphi\!\left(\frac{1}{r^{*}}\right)} \frac{1}{\underbrace{\varphi(\varphi\ldots\varphi\!\left(\frac{1}{r_2}\right))}_{m_0}}p_1^0\le \frac{r_2}{\varphi\!\left(\frac{1}{r^{*}}\right)}p_1^0=r^*p_1^0,$$
$$
p_{k+1}^{\,1}
= A p_{k+1}^{\,0}
\le A\!\left(r^{*} p_{1}^{\,0}\right)
\le \frac{1}{\varphi\!\left(\frac{1}{r^{*}}\right)}\, A p_{1}^{\,0}
= \frac{1}{\varphi\!\left(\frac{1}{r^{*}}\right)}\,Ax_{m_0}
=$$$$=  \frac{1}{\varphi\!\left(\frac{1}{r^{*}}\right)}\,x_{m_0+1}=
\frac{1}{\varphi\!\left(\frac{1}{r^{*}}\right)} p_1^{\,1}
\le r^{*} p_1^{\,1}.
$$
$$
\dots\dots\dots\dots\dots\dots\dots\dots\dots\dots\dots\dots\dots\dots\dots\dots\dots\dots\dots\dots\dots
$$
$$
p_{k+1}^{\,m_0-1}
= A p_{k+1}^{\,m_0-2}
\le A\!\left(r^{*} p_{1}^{\,m_0-2}\right)
\le \frac{1}{\varphi\!\left(\frac{1}{r^{*}}\right)}\, A p_{1}^{\,m_0-2}
= \frac{1}{\varphi\!\left(\frac{1}{r^{*}}\right)}\,x_{2m_0-1}= $$$$= \frac{1}{\varphi\!\left(\frac{1}{r^{*}}\right)}\, p_1^{\,m_0-1}
\le r^{*} p_1^{\,m_0-1}.
$$
Further, performing arguments analogous to those in the proof
of Theorem~\ref{thm3}, from \eqref{Khachatryan52} and \eqref{Khachatryan55}
we obtain that there exist positive constants
$C_1,\dots,C_{m_0} > 0$ and $q_1,\dots,q_{m_0} \in (0,1)$ such that
\begin{align}\label{Khachatryan56}&
\| p_k^{\,0} - p_{k+1}^{\,0} \|
\le C_1\, q_1^{\,m_0(k-1)},
\\&\nonumber
\| p_k^{\,1} - p_{k+1}^{\,1} \|
\le C_2\, q_2^{\,m_0(k-1)+1},\\
&\nonumber \ldots,
\\&\nonumber
\| p_k^{\,m_0-1} - p_{k+1}^{\,m_0-1} \|
\le C_{m_0}\, q_{m_0}^{\,m_0(k-1)+m_0-1},
\quad k = 1,2,\dots,
\end{align}
Moreover, the following limits exist:
\begin{equation}\label{Khachatryan57}
\lim_{k\to\infty} p_k^{\,0} = :p_0,\quad
\lim_{k\to\infty} p_k^{\,1} = :p_1,\quad \dots \quad,
\lim_{k\to\infty} p_k^{\,m_0-1} = :p_{m_0-1}.
\end{equation}

These limits satisfy the following inequalities:
\begin{equation}\label{Khachatryan58}
r^*p_1^0\ge p_0\ge \tau^*p_1^0, \ r^*p_1^1\ge p_1\ge \tau^*p_1^1, \ldots, r^*p_1^{m_0-1}\ge p_{m_0-1}\ge \tau^*p_1^{m_0-1}.
\end{equation}

From \eqref{Khachatryan57} and \eqref{Khachatryan54}, taking into account
the continuity of the operator $A$, it immediately follows that
\begin{equation}\label{Khachatryan59}
p_0 = A p_{m_0-1}, \quad
p_1 = A p_0, \quad
p_2 = A p_1, \quad \dots \quad,
p_{m_0-1} = A p_{m_0-2}.
\end{equation}

From \eqref{Khachatryan59} it follows immediately that
$
A^{m_0} p_j = p_j, \ j = 0,1,\dots,m_0-1 .
$

From the closedness of the cone $K$, by virtue of
\eqref{Khachatryan55}, \eqref{Khachatryan48}, \eqref{Khachatryan51},
it follows that
$
p_j \in K,\quad p_j \ne \theta,
\ j = 0,1,\dots,m_0-1,
$
since $v_0 \in K$, $v_0 \ne \theta$ and
$
A^n v_0 \in K,\ A^n v_0 \ne \theta,
\ n = 1,2,\dots .
$

Suppose now that there exist numbers
$d_1 \in (0,1)$, $d_2 \in (1,+\infty)$
and indices
$i_0,j_0 \in \{0,1,\dots,m_0-1\}$,
$i_0 \ne j_0$, with
$\gcd(m_0,|i_0-j_0|)=1$
such that condition~\eqref{Khachatryan47} holds.

Then, taking into account \eqref{Khachatryan48} and \eqref{Khachatryan51},
condition~\eqref{Khachatryan47} can be rewritten as
\begin{equation}\label{Khachatryan60}
d_1\, p^{i_0}_{\,1} \le p^{j_0}_{\,1} \le d_2\, p^{i_0}_{\,1}.
\end{equation}

From \eqref{Khachatryan58} and \eqref{Khachatryan60} it follows that
there exist numbers
$\ell_1 \in (0,1)$ and $\ell_2 > 1$ such that
\begin{equation}\label{Khachatryan61}
\ell_1\, p_{i_0} \le p_{j_0} \le \ell_2\, p_{i_0}.
\end{equation}

Again, using the monotonicity and strong concavity of the operator $A$,
and by virtue of equalities \eqref{Khachatryan59}, from \eqref{Khachatryan61}
we obtain the following chain of inequalities:
\begin{equation}\label{Khachatryan62}
\underbrace{\varphi\!\bigl(\varphi(\dots \varphi(\ell_1))\bigr)}_{m_0k}\, p_{i_0}
\le p_{j_0}
\le
\frac{1}{\underbrace{\varphi\!\bigl(\varphi(\dots \varphi\left(1/\ell_2\right))\bigr)}_{m_0k}}\, p_{i_0},\quad k=1,2,\ldots.
\end{equation}
Since
$
\lim\limits_{k\to\infty}
\underbrace{\varphi\!\bigl(\varphi(\dots \varphi(\ell_1))\bigr)}_{m_0k}
=
\lim\limits_{k\to\infty}
\underbrace{\varphi\!\bigl(\varphi(\dots \varphi(1/\ell_2))\bigr)}_{m_0k}
= 1,
$
it follows from \eqref{Khachatryan62} that
$p_{i_0} = p_{j_0}$.

Since $\gcd(m_0,|i_0-j_0|)=1$, by virtue of \eqref{Khachatryan59}
from equality $p_{i_0}=p_{j_0},$ it follows that
$
p_0 = p_1 = \dots = p_{m_0-1} =: p^{*}
\quad \text{and} \quad
A p^{*} = p^{*}.
$

Thus, the theorem is proved.
\end{proof}

\medskip
\begin{remark}
It should be noted that in Theorem~\ref{thm6} the condition \linebreak
$\gcd(m_0,|i_0-j_0|)=1$
is essential. Indeed, if the condition doesn't hold,
from \eqref{Khachatryan59} and equality $p_{i_0}=p_{j_0}$
it would not follow that
$
p_0 = p_1 = \dots = p_{m_0-1}.
$ If we consider, for example, the case
$m_0=6$, $i_0=3$, $j_0=1$, then
$\gcd(m_0,|i_0-j_0|)=2$,
and from equality $p_3=p_1$ and relations
$
p_0 = A p_5,\
p_1 = A p_0,\
p_2 = A p_1,\
p_3 = A p_2,\
p_4 = A p_3,\
p_5 = A p_4
$
only these equalities would follow:
$
p_0 = p_2 = p_4
$
and
$
p_1 = p_3 = p_5
$
(see Fig.~\ref{fig1}).
\begin{figure}[h]
\centering
\includegraphics[width=0.4\textwidth]{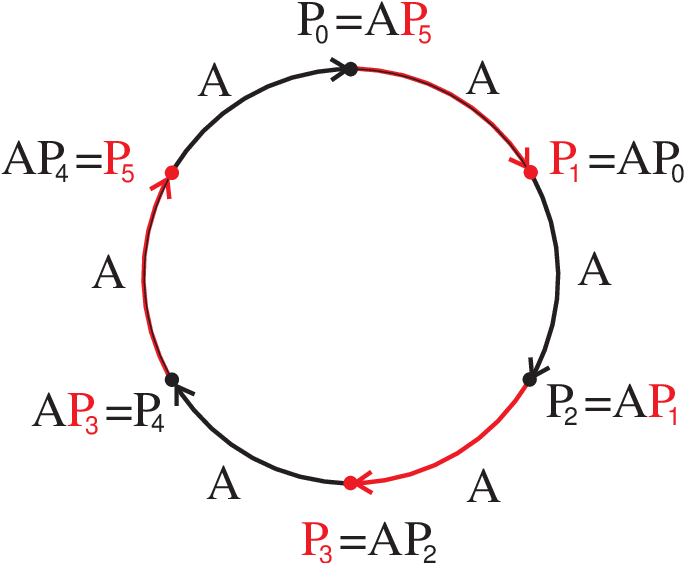}
\caption{Cyclic transfer of fixed points.}\label{fig1}
\end{figure}
\end{remark}
\newpage
\section{Uniqueness theorems for fixed points. Non-uniqueness of a fixed point for continuous non-concave and non-monotone operators}\label{sec4}

The following result holds.
\begin{theorem}\label{thm7}
Let $A : K \to K$ be a monotone and strongly concave operator
in the cone $K$.
Assume that there exists $x^* \in K$, $x^* \neq \theta$
such that
$
A x^* = x^* .
$
Then for any $r_1 \in (0,1)$ and $r_2 > 1$
in the conical segment
$\langle r_1 x^*, r_2 x^* \rangle$
the operator $A$  has no fixed points other than $x^*.$
\end{theorem}
\begin{proof}
Assume the contrary: there exists
$\tilde{x} \in \langle r_1 x^*, r_2 x^* \rangle$
$\tilde{x} \neq x^*$,
such that
$
A \tilde{x} = \tilde{x}.
$
Taking into account the monotonicity and strong concavity
of the operator $A$ as well as using the equalities
$A \tilde{x} = \tilde{x}$, $A x^* = x^*$,
one can easily prove by mathematical induction
that the following two–sided inequality holds:
\begin{equation}\label{Khachatryan63}
\underbrace{\varphi\!\bigl(\varphi(\dots \varphi(r_1))\bigr)}_n\, x^*
\;\le\;
\tilde{x}
\;\le\;
\frac{1}{\underbrace{\varphi\!\bigl(\varphi(\dots \varphi(1/r_2)\dots)\bigr)}_n}\, x^*,
\quad n = 1,2,\dots .
\end{equation}
Since
$
\lim\limits_{n \to \infty}
\underbrace{\varphi\!\bigl(\varphi(\dots \varphi(r_1))\bigr)}_n
=
\lim\limits_{n \to \infty}
\underbrace{\varphi\!\bigl(\varphi(\dots \varphi(1/\tau_2))\bigr)}_n
= 1,
$
due to the closedness of cone $K$, from \eqref{Khachatryan63} we obtain
that $\tilde{x} = x^*$.
Thus the theorem is proved.
\end{proof}

\medskip
We now formulate a uniqueness theorem for a fixed point
for the operator $A + A_0$, where $A$ and $A_0$ satisfy the monotonicity condition, with $A$ being a strongly concave operator and $A_0$ being concave:

\medskip
\begin{theorem}\label{thm8}
Let $A : K \to K$ be a monotone and strongly concave operator
in the cone $K$.
Further, let $A_0 : K \to K$ be a monotone and critical operator,
and suppose
$
(A + A_0)\tilde{x} = \tilde{x}
$
where $\tilde{x} \in K$, $\tilde{x} \neq \theta$.
If there exists a constant $C_0 > 0$ such that
$
A_0 u \le C_0 A u
$
and
$
A_0(tu) \ge t A_0 u,
\ t \in [0,1],\ u \in K,
$
then for any $\alpha_1 \in (0,1)$, $\alpha_2 > 1, $
the operator $A + A_0$ has no fixed points other than $\tilde{x}$ in the conical segment
$\langle \alpha_1 \tilde{x}, \alpha_2 \tilde{x} \rangle.$
\end{theorem}

\medskip
The proof follows immediately from the strong concavity of the operator $A + A_0$
(see the proof of Theorem~\ref{thm5}),
the monotonicity of the operators $A$, $A_0$,
and from Theorem~\ref{thm7}.
\hfill $\qed$

\medskip
\begin{remark}
It is interesting to note that in Theorems~\ref{thm7} and~\ref{thm8}
we did not require the operators to satisfy the continuity condition,
and the corresponding cone was arbitrary.
\end{remark}

Now, using monotone continuous and concave operators,
we construct continuous and critical operators
that do not possess the properties of concavity and monotonicity,
and for which the set of fixed points has the cardinality of the continuum.

For this purpose we first introduce several definitions
and notations needed for the formulation of corresponding theorems.

Let $x^* \in K$, $x^* \neq \theta$ be an arbitrary element.
Consider the following open ball in our Banach space $E$:
$$
B := \{ y \in E : \|y - x^*\| < \|x^*\| \}.
$$
The set
$
T_{x^*} := B \cap K
$
will be called the \emph{local section of the cone $K$ at the point $x^*$}.
Define the following operator on the cone $K$:
\begin{equation}\label{Khachatryan64}
\pi(x) :=
\begin{cases}
x^*, & \text{if } x \in T_{x^*}, \\[1ex]
x + \dfrac{\|x^*\|}{\|x - x^*\|}\,(x^* - x), & \text{if } x \in K \setminus T_{x^*}.
\end{cases}
\end{equation}

Note that the operator $\pi$ is continuous
since
the values of the operator $\pi$ ``from both sides'' coincide with $x^*$ on the boundary
$
S := \{ x \in K : \|x - x^*\| = \|x^*\| \}.
$

Let us verify that $\pi : K \to K$.
Indeed, let $x \in K$ be arbitrary.
If $x \in T_{x^*}$, then $\pi(x) = x^* \in K$.
If $x \in K \setminus T_{x^*}$, then due to the convexity of the cone $K$,
the inequality
$
0 < \frac{\|x^*\|}{\|x - x^*\|} \le 1
$
and formula \eqref{Khachatryan64} imply
$$
\pi(x)
= \frac{\|x^*\|}{\|x - x^*\|}\,x^*
+ \left( 1 - \frac{\|x^*\|}{\|x - x^*\|} \right) x
\in K.
$$

Now let $A : K \to K$ be a monotone and continuous operator in the cone $K$,
such that
$
A(tu) \ge t A u,
\ u \in K,\ t \in [0,1].
$
Assume that $A$ has a unique fixed point
$x^* \in K$, $x^* \neq \theta$:
$
A x^* = x^*.
$
We show that there exists a continuous and critical operator
$
\widetilde{A} := \widetilde{A}(x^*, A),
$
acting in the cone $K$, whose set of fixed points coincides with the closure of the local section of the cone $K$ at the point $x^*$.

Indeed, define the operator $\widetilde{A}$ as follows:
\begin{equation}\label{Khachatryan65}
\widetilde{A}x := x + A\bigl(\pi(x)\bigr) - \pi(x),
\quad x \in K.
\end{equation}

From \eqref{Khachatryan64} and \eqref{Khachatryan65} it follows that all elements of the set
$\overline{T}_{x^*}$ are exactly the set of fixed points of the operator
$\widetilde{A}$, since
$
\widetilde{A}x
= x + A x^* - x^*
= x,
\ x \in \overline{T}_{x^*}.
$

Clearly, the set $\overline{T}_{x^*}$ has the cardinality of the continuum,
since, for example, for all $t \in [0,1]$ the elements $t x^* \in T_{x^*}$.

Let us now verify that $\widetilde{A} : K \to K$.
If $x \in T_{x^*}$, then as already noted,
$
\widetilde{A}x = x \in K.
$
If $x \in K \setminus  T_{x^*}$, then taking into account consideration
the monotonicity of the operator $A$ and the inequality
$
A(tu) \ge t A u,
\ t \in [0,1],\; u \in K,
$
we obtain
 \begin{align*}&
\widetilde{A}x
= x + A\!\left(
\frac{\|x^*\|}{\|x - x^*\|} x^*
+ \left(1 - \frac{\|x^*\|}{\|x - x^*\|}\right) x
\right)-
 \\& \nonumber
- \frac{\|x^*\|}{\|x - x^*\|} x^*
- \left(1 - \frac{\|x^*\|}{\|x - x^*\|}\right) x\ge
 \\& \nonumber
\ge A\!\left(
\frac{\|x^*\|}{\|x - x^*\|} x^*
\right)
- \frac{\|x^*\|}{\|x - x^*\|} x^*+\frac{\|x^*\|}{\|x - x^*\|} x\ge
 \\& \nonumber
\ge \frac{\|x^*\|}{\|x - x^*\|} A x^*
- \frac{\|x^*\|}{\|x - x^*\|} x^*+\frac{\|x^*\|}{\|x - x^*\|} x=
\frac{\|x^*\|}{\|x - x^*\|} x
\ge \theta,
 \end{align*}
since
$
\dfrac{\|x^*\|}{\|x - x^*\|} \in (0,1],
\ x \in K \setminus T_{x^*}.
$

Now note that if $x \in K \setminus \overline{T}_{x^*}$,
then $\widetilde{A}x \neq x$.
Indeed, suppose that there exists $x_0 \in K \setminus \overline{T}_{x^*}$ such that
$
\widetilde{A} x_0 = x_0 .
$
First,
$
\frac{\|x^*\|}{\|x_0 - x^*\|} \in (0,1),
$
and from \eqref{Khachatryan65} together with \eqref{Khachatryan64} we obtain
$$
A\!\left(
\frac{\|x^*\|}{\|x_0 - x^*\|} x^*
+ \left(1 - \frac{\|x^*\|}{\|x_0 - x^*\|}\right) x_0
\right)=$$$$=
\frac{\|x^*\|}{\|x_0 - x^*\|} x^*
+ \left(1 - \frac{\|x^*\|}{\|x_0 - x^*\|}\right) x_0 .
$$

Since
$
\frac{\|x^*\|}{\|x_0 - x^*\|} x^*
+ \left(1 - \frac{\|x^*\|}{\|x_0 - x^*\|}\right) x_0
\in K,
$
taking into account the uniqueness of the nonzero fixed point of the operator $A$
in the cone $K$, we obtain
$
\frac{\|x^*\|}{\|x_0 - x^*\|} x^*
+ \left(1 - \frac{\|x^*\|}{\|x_0 - x^*\|}\right) x_0
= x^*.
$
Hence, since $\frac{\|x^*\|}{\|x_0 - x^*\|}\in(0,1)$, we obtain $x_0 = x^*,$
which is impossible because $x^* \in \overline{T}_{x^*},$ whereas $x_0 \in K \setminus \overline{T}_{x^*}$.

Thus, based on the above, we arrive at the following result.
\begin{theorem}\label{thm9}
Let $A : K \to K$ be a monotone and continuous operator
in the cone $K$, such that
$
A(tu) \ge t A u,
\ u \in K,\; t \in [0,1].
$
If the operator $A$ has a unique fixed point
$x^* \in K$, $x^* \ne \theta$, then the following statements hold:
\begin{enumerate}
\item[1)]
the operator
$
\widetilde{A} := \widetilde{A}(x^*,A),
$
defined by means of \eqref{Khachatryan64}, \eqref{Khachatryan65},
is continuous and critical, and
$\widetilde{A} : K \to K$;
\item[2)]
the set of fixed points of the operator $\widetilde{A}$ coincides
with the closure of the local section of the cone $K$ at the point $x^*$.
\end{enumerate}\hfill\qed
\end{theorem}
Now for any number $\lambda \in (0,1)$ define the following
subsets of the conical segment $\langle \theta, x^* \rangle$:
$$
D_{\lambda,x^*}
:=
\bigl\{
x \in \langle \theta, x^* \rangle :
\|x - x^*\| > \lambda \|x^*\|
\bigr\}.
$$
Clearly $D_{\lambda,x^*} \ne \varnothing$, since, for example,
$\dfrac{1-\lambda}{2} x^* \in D_{\lambda,x^*}$.

Now
define the following operators on the conical segment $\langle \theta, x^* \rangle:$
\begin{equation}\label{Khachatryan66}
\mathcal{P}_{\lambda}(x) :=
\begin{cases}
x^* - x, & \text{if } x \in D_{\lambda,x^*},\\[6pt]
\dfrac{\|x - x^*\|}{\lambda \|x^*\|}(x^* - x),
& \text{if } x \in \langle \theta, x^* \rangle \setminus D_{\lambda,x^*},
\end{cases}
\end{equation}
\begin{equation}\label{Khachatryan67}
\widehat{A}x
:=
A\bigl(\mathcal{P}_{\lambda}(x) + x\bigr)
-
\mathcal{P}_{\lambda}(x),
\qquad
\lambda \in (0,1),
\;
x \in \langle \theta, x^* \rangle.
\end{equation}

Proceeding analogously to the proof of Theorem~\ref{thm9},
we obtain the following theorem.

\medskip
\begin{theorem}\label{thm10}
Under assumptions of Theorem~\ref{thm9} the operator
$\widehat{A} := \widehat{A}(x^*,A)$,
defined by means of \eqref{Khachatryan66}, \eqref{Khachatryan67},
is continuous and critical, and
$
\widehat{A} :
\langle \theta, x^* \rangle
\to
\langle \theta, x^* \rangle .
$
Moreover, for any $\lambda \in (0,1)$
the closure of the set $D_{\lambda,x^*}$
coincides with the set of fixed points of the operator $\widehat{A}$.\hfill\qed
\end{theorem}

\medskip
\begin{remark}
By direct verification, one can see that, generally speaking,
the operators $\widetilde{A}$ and $\widehat{A}$ constructed above
do not possess the property of concavity.
We assume that the presence of an infinite set
of fixed points for such operators
is primarily due to the absence
of the concavity (or strong concavity)
property for the operators $\widetilde{A}$ and $\widehat{A}$.
We note that the monotonicity property of
$\widetilde{A}$ and $\widehat{A}$
is yet another important property
that the operators do not possess.
\end{remark}
\section{Applications}\label{sec5}

\subsection{Application of the obtained results
to the dynamic theory of $p$-adic strings}\label{subsec5.1}

Let $B$ be an $n$-dimensional elliptic operator of the following form:
$$
B :=
\beta_0 \frac{\partial^2}{\partial t^2}
+ \beta_1 \frac{\partial^2}{\partial x_1^2}
+ \cdots
+ \beta_{n-1} \frac{\partial^2}{\partial x_{n-1}^2},
$$
where $\beta_j > 0$, $j = 0,1,\ldots,n-1$, are constants.
Consider the following nonlinear pseudodifferential equation
in $\mathbb{R}^n$:
\begin{equation}\label{Khachatryan68}
\bigl(e^{B} f\bigr)
\bigl(t,x_1,\ldots,x_{n-1}\bigr)
=
f^{p}\bigl(t,x_1,\ldots,x_{n-1}\bigr),
\
(t,x_1,\ldots,x_{n-1}) \in \mathbb{R}^n,
\end{equation}
with respect to the unknown real-valued function
$f(t,x_1,\ldots,x_{n-1})$,
where $p>1$ is an odd number.
Equation \eqref{Khachatryan68} appears in the dynamic theory of
$p$-adic strings for the scalar tachyonic field,
where $f$ is the tachyonic field for open strings
(see~\cite{bre5}, \cite{vla6}).
It should be noted that $p$-adic strings are of interest
as a certain approximation to nonlocal theories
arising from string field theory
(see~\cite{aref10}).
We will now assign precise meaning to equation \eqref{Khachatryan68}.
Equation \eqref{Khachatryan68} represents
a pseudodifferential equation
with a symbol of the form
$
\exp\!\left(
- \sum\limits_{j=0}^{n-1} e^{\beta_j} \varkappa_j^2
\right),
$ $
\varkappa_j \in \mathbb{R},
\ j=0,1,\ldots,n-1,
$
which for positive $\beta_j$,
$j=0,1,\ldots,n-1$,
can be represented
as a nonlinear integral equation
(see, for example,~\cite{vla6}),
\begin{align}\label{Khachatryan69}&
\frac{1}{(2\pi)^{\,n-1}}
\int\limits_{-\infty}^{\infty}\dots
\int\limits_{-\infty}^{\infty}
\widehat{f}(\tau,\vec{\ell})\,
H_{\beta_0}(t-\tau)\cdot
\\&\nonumber\cdot
\exp\!\left(
-\sum\limits_{j=1}^{n-1}\beta_j \ell_j^{\,2}
- i(\vec{x},\vec{\ell})
\right)
\,d\tau d\ell_1 \cdots d\ell_{n-1}=\\&\nonumber
=f^{p}(t,x_1,\ldots,x_{n-1}),\quad (t,x_1,\ldots,x_{n-1})\in\mathbb{R}^n,
\end{align}
where
\begin{align}\label{Khachatryan70}&
\vec{x}:=(x_1,\ldots,x_{n-1}),\quad
\vec{\ell}:=(\ell_1,\ldots,\ell_{n-1}),\quad (\vec{x},\vec{\ell}) := \sum_{j=1}^{n-1} x_j \ell_j,
\\&\nonumber
 H_{\beta_j}(u)
:=
\frac{1}{\sqrt{4\pi\beta_j}}
\exp\!\left(-\frac{u^2}{4\beta_j}\right),\;
u\in\mathbb{R},
\;
j=0,1,\ldots,n-1,
\end{align}
and
\begin{equation}\label{Khachatryan71}
\widehat{f}\left(t,\vec{\ell}\right)
:=
\int\limits_{-\infty}^{\infty} \!\!\cdots
\int\limits_{-\infty}^{\infty}
f\left(t,\vec{y}\right)\,
\exp\left(i\left(\vec{y},\vec{\ell}\right)\right)
\,dy_1\cdots dy_{n-1}.
\end{equation}

A solution $f(t,x_1,\ldots,x_{n-1})$ of equation
\eqref{Khachatryan68} (or equation \eqref{Khachatryan69}) should be sought
in the class of tempered generalized functions
$S'(\mathbb{R}^n)\subset\mathcal{D}'(\mathbb{R}^n)$
(see~\cite{gel11},\cite{Io} for details on the spaces
$S'(\mathbb{R}^n)$ and $\mathcal{D}'(\mathbb{R}^n)$).
It is noteworthy to point out that Hadamard-Bergman-type linear convolution operators have been thoroughly researched in the class of tempered generalized functions in \cite{KV}.

Taking into account \eqref{Khachatryan70}, \eqref{Khachatryan71}, after simple
but cumbersome calculations from \eqref{Khachatryan69} we obtain
the following nonlinear integral equation:
\begin{equation}\label{Khachatryan72}
\begin{array}{c}
\displaystyle f^{p}(t,x_1,\ldots,x_{n-1})
=\\
\displaystyle=\int\limits_{-\infty}^{\infty} \!\!\cdots
\int\limits_{-\infty}^{\infty}
H_{\beta_0}(t-\tau)
\prod_{j=1}^{n-1}
H_{\beta_j}(x_j-y_j)
f(\tau,y_1,\ldots,y_{n-1})
\, d\tau\,dy_1\cdots dy_{n-1},
\\
\displaystyle(t,x_1,\ldots,x_{n-1})\in\mathbb{R}^n .
\end{array}
\end{equation}

By direct verification, one can see that if
$\varphi(t,x_1,\ldots,x_{n-1})$
is a continuous function on
$
\mathbb{R}_+^{\,n}
=
\underbrace{\mathbb{R}_+\times\cdots\times\mathbb{R}_+}_{n},
\
\mathbb{R}_+ := [0,+\infty),
$
which solves the equation
\begin{equation}\label{Khachatryan73}
\begin{array}{c}
\displaystyle\varphi^p(t,x_1,\ldots,x_{n-1})
=\\
\displaystyle=\int\limits_{0}^{\infty}\!\!\cdots\!\!
\int\limits_{0}^{\infty}
\Bigl(
H_{\beta_0}(t-\tau)-H_{\beta_0}(t+\tau)
\Bigr)
\prod_{j=1}^{n-1}
\Bigl(
H_{\beta_j}(x_j-y_j)
-
H_{\beta_j}(x_j+y_j)
\Bigr)\times
\\
\quad\times
\varphi(\tau,y_1,\ldots,y_{n-1})
\, d\tau\,dy_1\cdots dy_{n-1},\quad
(t,x_1,\ldots,x_{n-1})\in\mathbb{R}_+^{\,n},
\end{array}
\end{equation}
and

\begin{enumerate}
\item[1)]
$
\varphi_1(t,x_1,\ldots,x_{n-1})
=
\begin{cases}
\varphi(t,x_1,\ldots,x_{n-1}),
& t\ge 0,\ (x_1,\ldots,x_{n-1})\in\mathbb{R}_+^{\,n-1},
\\[2mm]
-\varphi(-t,x_1,\ldots,x_{n-1}),
& t<0,\ (x_1,\ldots,x_{n-1})\in\mathbb{R}_+^{\,n-1},
\end{cases}
$

\item[2)]
$
\varphi_2(t,x_1,\ldots,x_{n-1})
=
\begin{cases}
\varphi_1(t,x_1,\ldots,x_{n-1}),
& t\in\mathbb{R},\ x_1\ge 0,\\ &(x_2,\ldots,x_{n-1})\in\mathbb{R}_+^{\,n-2},
\\[2mm]
-\varphi_1(t,-x_1,x_2,\ldots,x_{n-1}),
& t\in\mathbb{R},\ x_1<0,\\ &(x_2,\ldots,x_{n-1})\in\mathbb{R}_+^{\,n-2},
\end{cases}
$

$$
\ldots\ldots\ldots\ldots\ldots\ldots\ldots\ldots\ldots\ldots\ldots\ldots\ldots\ldots\ldots\ldots\ldots\ldots\ldots\ldots
$$

\item[$n-1$)]
$
\varphi_{n-1}(t,x_1,\ldots,x_{n-1})
=
\begin{cases}
\varphi_{n-2}(t,x_1,\ldots,x_{n-1}),
\qquad x_{n-2},x_{n-1}\ge 0,\\ \qquad\qquad\qquad\qquad\qquad(t,x_1,\ldots,x_{n-3})\in\mathbb{R}^{\,n-2},
\\[2mm]
-\varphi_{n-2}(t,x_1,\ldots x_{n-3},-x_{n-2},x_{n-1}),  x_{n-2}<0, x_{n-1}\ge0,
\\[2mm]
\qquad\qquad\qquad\qquad\qquad(t,x_1,\ldots,x_{n-2})\in\mathbb{R}^{\,n-3},
\end{cases}
$
\end{enumerate}

then the odd extension of the function
$\varphi_{n-1}(t,x_1,\ldots,x_{n-1})$
with respect to the variable $x_{n-1}$ onto $(-\infty,0)$:
$$
f(t,x_1,\ldots,x_{n-1})=
\begin{cases}
\varphi_{n-1}(t,x_1,\ldots,x_{n-1}),\quad x_{n-1}\in\mathbb{R}^+,\\
\qquad\qquad\qquad\qquad\qquad (t,x_1,\ldots,x_{n-2})\in\mathbb{R}^{\,n-1},
\\[2mm]
-\varphi_{n-1}(t,x_1,\ldots,x_{n-2},-x_{n-1}),\quad \ x_{n-1}<0,\\
\qquad\qquad\qquad\qquad\qquad (t,x_1,\ldots,x_{n-2})\in\mathbb{R}^{\,n-1},
\end{cases}
$$
is a continuous solution of equation \eqref{Khachatryan72} on $\mathbb{R}^n.$

We seek a nonnegative, nontrivial, bounded and continuous
solution on $\mathbb{R}_+^{\,n}$ of equation \eqref{Khachatryan73}.

Now, let us consider the following nonlinear integral operator:
\begin{equation}\label{Khachatryan74}
\begin{array}{c}
\displaystyle(A\psi)(t,x_1,\ldots,x_{n-1})
:=
\int\limits_{0}^{\infty}\!\!\cdots\!\!\int\limits_{0}^{\infty}
\bigl(
H_{\beta_0}(t-\tau)-H_{\beta_0}(t+\tau)
\bigr)\times
\\
\displaystyle\times
\prod_{j=1}^{n-1}
\bigl(
H_{\beta_j}(x_j-y_j)-H_{\beta_j}(x_j+y_j)
\bigr)
\psi^\alpha(\tau,y_1,\ldots,y_{n-1})
\,d\tau\,dy_1\cdots dy_{n-1},\\
\displaystyle (t,x_1,\ldots,x_{n-1})\in\mathbb{R}_+^{\,n},
\end{array}
\end{equation}
where
$
\alpha := \dfrac{1}{p}\in(0,1).
$

Let $C_B(\mathbb{R}_+^{\,n})$ be the space of continuous and bounded
functions on $\mathbb{R}_+^{\,n}$.
Consider the cone of nonnegative functions in the space
$C_B(\mathbb{R}_+^{\,n})$:
$$
K :=
\bigl\{
f\in C_B(\mathbb{R}_+^{\,n}) : f\ge 0
\bigr\}.
$$
Obviously, $K$ is a normal cone.
Taking into account \eqref{Khachatryan70} and \eqref{Khachatryan74}, one easily verifies that
$
A:K\to K.
$

Since
$
H_{\beta_j}(x-y)\ge H_{\beta_j}(x+y),
\ j=0,1,\ldots,n-1,\ x,y\in\mathbb{R}_+,
$
it also follows from \eqref{Khachatryan74} that the operator $A$ is monotone.
Let us verify that the operator $A$ is strongly concave in the cone $K$.
Indeed, choosing as the mapping $\varphi$
$
\varphi(\sigma)=\sigma^{\gamma},
$
where $\sigma\in[0,1]$, $\gamma\in[\alpha,1)$, and taking into account the inequality
$H_{\beta_j}(x-y)\ge H_{\beta_j}(x+y)$,
$j=0,1,\ldots,n-1$, $x,y\in\mathbb{R}_+$,
from \eqref{Khachatryan74} we obtain
$$
(A\sigma\psi)(t,x_1,\ldots,x_{n-1})
=
\sigma^{\alpha}
(A\psi)(t,x_1,\ldots,x_{n-1})
\ge
\sigma^{\gamma}(A\psi)(t,x_1,\ldots,x_{n-1}),
$$
\[
(t,x_1,\ldots,x_{n-1})\in\mathbb{R}_+^{\,n},
\qquad
\sigma\in[0,1],\quad \psi\in K.
\]

Choosing $v_0\equiv1\in K$ and $n_0=2$
as the element and the index respectively,
let us verify condition \eqref{Khachatryan4} of Theorem~\ref{thm1}.
For this purpose, for an arbitrary $\varepsilon\in(0,1)$
consider the following characteristic equations:
\begin{equation}\label{Khachatryan75}
\int\limits_{0}^{\infty} H_{\beta_j}(u)e^{-s u}\,du
=
\frac{\varepsilon}{2},
\qquad
j=0,1,\ldots,n-1,
\end{equation}

with respect to the unknown $s\ge0$.
From \eqref{Khachatryan70} it immediately follows that the characteristic equations \eqref{Khachatryan75}
have unique positive solutions
$s_j:=s_j(\varepsilon)$,
$j=0,1,\ldots,n-1$.

From Lemma~1 of the paper~\cite{khach12} the following lower estimate holds:
\begin{equation}\label{Khachatryan76}
\begin{array}{c}
\displaystyle\int\limits_{0}^{\infty}
\bigl(
H_{\beta_j}(x-y)-H_{\beta_j}(x+y)
\bigr)
\bigl(1-e^{-s_j y}\bigr)\,dy
\ge\\
\displaystyle\ge
\varepsilon\bigl(1-e^{-s_j x}\bigr),\quad
x\in\mathbb{R}_+,\quad j=0,1,\ldots,n-1.
\end{array}
\end{equation}

Now consider the following functions on the set $(0,+\infty)$:
$$
Q_j(x):=
\frac{\varepsilon\bigl(1-e^{-s_j x}\bigr)}
{1-2\displaystyle\int\limits_{x}^{\infty} H_{\beta_j}(y)\,dy},
\qquad
x>0,\quad j=0,1,\ldots,n-1.
$$

Taking into account \eqref{Khachatryan70}, \eqref{Khachatryan76} and applying L'Hôpital's rule,
we obtain
\begin{equation}\label{Khachatryan77}
Q_j(+\infty)=\varepsilon,
\quad
Q_j(x)>0,\ x>0,
\
Q_j\in C(0,+\infty),
\ j=0,1,\ldots,n-1,
\end{equation}
\begin{equation}\label{Khachatryan78}
\begin{array}{c}
\displaystyle Q_j(x)
\le
\frac{\varepsilon(1-e^{-s_j x})}
{\int\limits_{0}^{\infty}
\bigl(
H_{\beta_j}(x-y)-H_{\beta_j}(x+y)
\bigr)
\bigl(1-e^{-s_j y}\bigr)\,dy}
\le 1,\\
\displaystyle x>0,\quad j=0,1,\ldots,n-1,
\end{array}
\end{equation}
\begin{equation}\label{Khachatryan79}
Q_j(+0)
=
\lim_{x\to+0} Q_j(x)
=
\lim_{x\to+0}
\frac{\varepsilon s_j e^{-s_j x}}
{2 H_{\beta_j}(x)}
=
\frac{\varepsilon s_j}{2\sqrt{4\pi\,\beta_j}},
\
j=0,1,\ldots,n-1.
\end{equation}

From \eqref{Khachatryan77}, \eqref{Khachatryan78} and \eqref{Khachatryan79} it follows that
$Q_j\in C(\mathbb{R}_+)$,
$j=0,1,\ldots,n-1$,
and for each $j\in\{0,1,\ldots,n-1\}$
there exists $\tilde{\sigma}_j\in(0,1)$ such that
\begin{equation}\label{Khachatryan80}
Q_j(x)\ge \tilde{\sigma}_j,
\qquad
x\in\mathbb{R}_+.
\end{equation}

Taking into account \eqref{Khachatryan76}, \eqref{Khachatryan80}, \eqref{Khachatryan74} and the inequality
$H_{\beta_j}(x-y)\ge H_{\beta_j}(x+y)$,
$x,y\in\mathbb{R}_+$,
$j=0,1,\ldots,n-1$,
we obtain
$$
(A v_0)(t,x_1,\ldots,x_{n-1})
=
$$
$$
=
\int\limits_{0}^{\infty}\!\!\cdots\!\!\int\limits_{0}^{\infty}
\bigl(H_{\beta_0}(t-\tau)-H_{\beta_0}(t+\tau)\bigr)
\prod_{j=1}^{n-1}
\bigl(
H_{\beta_j}(x_j-y_j)-H_{\beta_j}(x_j+y_j)
\bigr)
\,d\tau\,dy_1\cdots dy_{n}\ge
$$
$$
\ge
\int\limits_{0}^{\infty}\!\!\cdots\!\!\int\limits_{0}^{\infty}
\bigl(H_{\beta_0}(t-\tau)-H_{\beta_0}(t+\tau)\bigr)(1-e^{-s_0\tau})\times
$$
$$
\times\prod_{j=1}^{n-1}
\bigl(
H_{\beta_j}(x_j-y_j)-H_{\beta_j}(x_j+y_j)
\bigr)
\bigl(1-e^{-s_j y_j}\bigr)
\,d\tau\,dy_1\cdots dy_{n}\ge
$$
$$
\ge\varepsilon^n (1-e^{-s_0t}) \prod_{j=1}^{n-1}\bigl(1-e^{-s_j x_j}\bigr),\quad (t,x_1,\ldots,x_{n-1})\in\mathbb{R}_+^n.
$$
From this, due to the monotonicity of the operator, we obtain
$$
(A^2 v_0)(t,x_1,\ldots,x_{n-1})
\ge $$$$\ge
\int\limits_{0}^{\infty}\!\!\cdots\!\!\int\limits_{0}^{\infty}
\bigl(H_{\beta_0}(t-\tau)-H_{\beta_0}(t+\tau)\bigr)
\prod_{j=1}^{n-1}
\bigl(
H_{\beta_j}(x_j-y_j)-H_{\beta_j}(x_j+y_j)
\bigr)\times
$$
$$
\times
\varepsilon^{n\alpha}
\bigl(1-e^{-s_0 \tau}\bigr)^{\alpha}
\prod_{j=1}^{n-1}
\bigl(1-e^{-s_j y_j}\bigr)^{\alpha}
\,d\tau\,dy_1\cdots dy_{n}\ge
$$
$$
\ge
\varepsilon^{n\alpha}
\int\limits_{0}^{\infty}\!\!\cdots\!\!\int\limits_{0}^{\infty}
\bigl(H_{\beta_0}(t-\tau)-H_{\beta_0}(t+\tau)\bigr)
\prod_{j=1}^{n-1}
\bigl(
H_{\beta_j}(x_j-y_j)-H_{\beta_j}(x_j+y_j)\bigr)\times
$$
$$
\times
\bigl(1-e^{-s_0 \tau}\bigr)
\prod_{j=1}^{n-1}
\bigl(1-e^{-s_j y_j}\bigr)
\,d\tau\,dy_1\cdots dy_{n}\ge
$$
$$
\ge \varepsilon^{n\alpha}\varepsilon^n(1-e^{-s_0t})\prod_{j=1}^{n-1}\bigl(1-e^{-s_j x_j}\bigr)\ge
$$
\[
\ge
\varepsilon^{\alpha n}
\prod_{j=0}^{n-1}\tilde{\sigma}_j
\int\limits_{0}^{\infty}\!\!\cdots\!\!\int\limits_{0}^{\infty}
\bigl(H_{\beta_0}(t-\tau)-H_{\beta_0}(t+\tau)\bigr)\times$$$$\times
\prod_{j=1}^{n-1}
\bigl(
H_{\beta_j}(x_j-y_j)-H_{\beta_j}(x_j+y_j)
\bigr)
\,d\tau\,dy_1\cdots dy_{n}=
\]
$$
= \varepsilon^{\alpha n}
\prod_{j=0}^{n-1}\tilde{\sigma}_j (Av_0)(t,x_1,\ldots,x_{n-1}).
$$

Thus,
\begin{equation}\label{Khachatryan81}
\begin{array}{c}
\displaystyle(A^2 v_0)(t,x_1,\ldots,x_{n-1})
\ge
\sigma_0\,(A v_0)(t,x_1,\ldots,x_{n-1}),\\ \displaystyle (t,x_1,\ldots,x_{n-1})\in\mathbb{R}_+^n,
\end{array}
\end{equation}
where
\[
\sigma_0
:=
\varepsilon^{\alpha n}
\prod_{j=0}^{n-1}\tilde{\sigma}_j
\in(0,1).
\]

On the other hand, from \eqref{Khachatryan70} it follows that
$
(A v_0)(t,x_1,\ldots,x_{n-1})\le 1= v_0,
$
whence, by monotonicity of operator $A,$ we obtain
\begin{equation}\label{Khachatryan82}
\begin{array}{c}
\displaystyle
(A^2 v_0)(t,x_1,\ldots,x_{n-1})
\le
(A v_0)(t,x_1,\ldots,x_{n-1}),
\\
\displaystyle
(t,x_1,\ldots,x_{n-1})\in\mathbb{R}_+^n.
\end{array}
\end{equation}

From \eqref{Khachatryan81}, \eqref{Khachatryan82} we conclude condition \eqref{Khachatryan4} for the operator \(A\)
defined according to \eqref{Khachatryan74}.
Thus, from Theorem~\ref{thm1} it follows that the operator \(A\) has a fixed point
\(\psi \in K\), \(\psi \neq 0\), and
$$
\tau^{*}(A v_0)(t,x_1,\ldots,x_{n-1})
\le
\psi(t,x_1,\ldots,x_{n})
\le
(A v_0)(t,x_1,\ldots,x_{n-1})
\le 1,
$$
\[
(t,x_1,\ldots,x_{n-1})\in\mathbb{R}_+^{n},
\]
where the number \(\tau^{*}\in(0,\sigma_0)\), and in this case
$
\tau^{*}=\sigma_0^{\frac{1}{1-\alpha}}.
$

From \eqref{Khachatryan73} and \eqref{Khachatryan74} it follows that the function
$
\varphi(t,x_1,\ldots,x_{n-1})
=:
\psi^{\alpha}(t,x_1,\ldots,x_{n-1})
$
is a solution of equation \eqref{Khachatryan73}.
Further, performing the procedure \(1)-n-1)\) and after that extending oddly
the function \linebreak \(\varphi_{n-1}(t,x_1,\ldots,x_{n-1})\) with respect to \(x_{n-1}\) onto \((-\infty,0)\),
we obtain a
sign-changing solution of equation \eqref{Khachatryan72} which is nontrivial, bounded, and continuous on \(\mathbb{R}^{n}\).
\medskip
\subsection{The nontrivial and bounded solution of nonlinear Urysohn-type integral
 equations in the critical case.}\label{subsec5.2}
On the set
$
\mathbb{R}: = (-\infty,+\infty),
$ consider the following class of nonlinear integral equations
of Urysohn type:
\begin{equation}\label{Khachatryan83}
f(x)
=
\int\limits_{-\infty}^{+\infty}
U(x,t,f(t))\,dt,
\qquad x\in\mathbb{R},
\end{equation}
with respect to the unknown nonnegative measurable and bounded function $f(x).$

The Urysohn kernel of equation \eqref{Khachatryan83} satisfies the following conditions:
\begin{enumerate}
  \item[$u_1)$] $U(x,t,0)\equiv0, (x,t)\in\mathbb{R}^2,$ and for every fixed \((x,t)\in\mathbb{R}^2\) the function \(U(x,t,z)\) is monotonically increasing in \(z\) on  \(\mathbb{R}_+\),
  \item[$u_2)$] There exists a number \(\eta>0\) such that
$
\int\limits_{-\infty}^{+\infty} U(x,t,\eta)\,dt \in C(\mathbb{R}),
$
$$
\sup\limits_{x\in\mathbb{R}}
\int\limits_{-\infty}^{+\infty} U(x,t,\eta)\,dt
=
\lim\limits_{|x|\to\infty}
\int\limits_{-\infty}^{+\infty} U(x,t,\eta)\,dt
=
\eta.
$$
  \item[$u_3)$] There exists a continuous monotonically increasing and concave
mapping
$
\varphi:[0,1]\to[0,1]
$
with the properties
$
\varphi(0)=0, \varphi(1)=1,$\linebreak $ \varphi'(0+)=+\infty,
$
such that
$$
U(x,t,\sigma z)\ge \varphi(\sigma)\,U(x,t,z),
\quad
\sigma\in[0,1],\;
(x,t)\in\mathbb{R}^2,\;
z\in\mathbb{R}_+.
$$
\end{enumerate}

Below, using Theorem~\ref{thm1}, we show that equation \eqref{Khachatryan83},
under conditions \(u_1)\)–\(u_3)\), possesses a positive and bounded
solution on \(\mathbb{R}\).

First, note that from condition \(u_1)\) it immediately follows
that a trivial (zero) solution of equation \eqref{Khachatryan83} exists.
We shall prove that besides such a trivial solution there exists a solution
\(f\) of equation \eqref{Khachatryan83} satisfying the inequality
\[
\tau^{*}\eta \le f(x) \le \eta, \quad x \in \mathbb{R},
\]
where $\tau^{*}\in(0,1),$ and it is determined from a certain characteristic
equation (see below).

Let $B(\mathbb{R})$ be the space of bounded functions on $\mathbb{R}$.
Denote by $K$ the cone of nonnegative functions in $B(\mathbb{R})$.
Obviously, the conic segment $\langle 0,\eta\rangle$ is a subset of $K.$

Consider the following nonlinear integral operator on $K$:
\begin{equation}\label{Khachatryan84}
(Af)(x)=\int\limits_{-\infty}^{+\infty} U(x,t,f(t))\,dt,
\quad x\in\mathbb{R},\ f\in K.
\end{equation}

Let us verify that $A:K\to K$. Indeed, let $f\in K$ be
an arbitrary function. Then, using conditions $u_1)$–$u_3)$, from \eqref{Khachatryan84}
we obtain
$$
0\le (Af)(x)\le \int\limits_{-\infty}^{+\infty}U(x,t,
\sup_{t\in\mathbb{R}} f(t))\,dt
= \int\limits_{-\infty}^{+\infty} U\left(x,t,
\frac{\sup\limits_{t\in\mathbb{R}} f(t)}{\eta}\,\eta\right)\,dt \leq
$$
$$
\le \int\limits_{-\infty}^{+\infty}
U\!\left(x,t,\max\left\{1,\frac{\sup\limits_{t\in\mathbb{R}} f(t)}{\eta}\right\}\eta\right)\,dt \le$$$$
\le \frac{1}{\varphi\!\left(
\dfrac{1}{\max\left\{1,\frac{\sup\limits_{t\in\mathbb{R}} f(t)}{\eta}\right\}}
\right)}
\int\limits_{-\infty}^{+\infty} U(x,t,\eta)\,dt \le
$$
$$
\le \eta \max\left\{1,\frac{\sup\limits_{t\in\mathbb{R}} f(t)}{\eta}\right\}
<+\infty,\quad x\in\mathbb{R}.
$$

From the monotonicity of the function $U(x,t,z)$ with respect to $z,$
it follows that the operator $A$ is monotone in the cone $K$.
Taking into account condition $u_3)$, from \eqref{Khachatryan84}
we obtain that
$
(A\sigma f)(x)\ge \varphi(\sigma)\,(Af)(x),
$ $ x\in\mathbb{R},\ \sigma\in[0,1],\ f\in K.
$

Finally, taking $v_0\equiv\eta$ and choosing $n_0=1$,
let us verify condition \eqref{Khachatryan4}.
Taking into account condition $u_2)$, from \eqref{Khachatryan84} we have
\begin{equation}\label{Khachatryan85}
(A v_0)(x)=\int\limits_{-\infty}^{+\infty} U(x,t,\eta)\,dt \le \eta=v_0,
\quad x\in\mathbb{R}.
\end{equation}

On the other hand, since conditions $u_1)$ and $u_2)$ hold,
there exists a number $r>0$ such that
\begin{equation}\label{Khachatryan86}
\int\limits_{-\infty}^{+\infty} U(x,t,\eta)\,dt \ge \frac{\eta}{2},
\quad |x|>r .
\end{equation}

Since
$
\int\limits_{-\infty}^{+\infty} U(x,t,\eta)\,dt \in C(\mathbb{R})
\quad \text{and} \quad
\int\limits_{-\infty}^{+\infty} U(x,t,\eta)\,dt>0,\ x\in\mathbb{R}
$
(see~conditions $u_1)$, $u_2)$), according to the Weierstrass theorem we obtain
\begin{equation}\label{Khachatryan87}
\mu := \min_{|x|\le r}
\int\limits_{-\infty}^{+\infty} U(x,t,\eta)\,dt >0 .
\end{equation}

Therefore, taking into account \eqref{Khachatryan86} and \eqref{Khachatryan87},
from \eqref{Khachatryan84} we obtain
\begin{equation}\label{Khachatryan88}
(A v_0)(x)\ge \frac{\min\left\{\mu,\frac{\eta}{2}\right\}}{\eta}
v_0,\quad x\in\mathbb{R}.
\end{equation}

Since
$
\sigma_0:= \dfrac{\min\{\mu,\frac{\eta}{2}\}}{\eta}\in\left(0,\frac12\right],
$
from \eqref{Khachatryan85} and \eqref{Khachatryan88} we conclude
that condition \eqref{Khachatryan4}
is satisfied for the operator $A$ defined by formula \eqref{Khachatryan84}.

Also note that the cone $K$ is normal, since the norm
is monotone in $K.$

Therefore, according to Theorem~\ref{thm1} equation \eqref{Khachatryan83}
possesses a positive solution $f,$ which is the uniform limit
of successive approximations
$
f_{n+1}(x)=(Af_n)(x), \ f_0(x)\equiv\eta,\ x\in\mathbb{R},
$
which converge at a rate of a geometric progression.

Moreover, the solution $f$ satisfies the two-sided inequality
$
\tau^*\,\eta \le f(x)\le \eta,\ x\in\mathbb{R},
$
where $\tau^*$ is uniquely determined from
the characteristic equation
$
\varphi(\tau)=\dfrac{\tau}{\sigma_0},
\ \tau^*\in(0,\sigma_0).
$

It is worth noting that, for various
particular representations of the Urysohn kernel $U,$ equation \eqref{Khachatryan83} has important
applications in the theory of nonlinear radiation transfer in spectral
lines and in the kinetic theory of gases within the framework of both the
classical and modified Bhatnagar–Gross–Krook models
(see \cite{eng13}-\cite{khach15}).

Using Theorem~\ref{thm7}, one can also assert that the solution of equation
\eqref{Khachatryan83} is unique in the conic segment
$
\langle \tau^* \eta,\eta\rangle .
$
\subsection{Application to a Cauchy problem
for a semilinear heat equation}\label{subsec5.3}

Consider the following Cauchy problem for a semilinear heat equation:
\begin{equation}\label{Khachatryan89}
\frac{\partial u}{\partial t}-\Delta u
=
\lambda(\text{x},t)\,G(u),
\quad \text{x}\in\mathbb{R}^n,\ t>0,
\end{equation}
\begin{equation}\label{Khachatryan90}
u(\text{x},0)=u_0(\text{x}),
\quad \text{x}\in\mathbb{R}^n,
\end{equation}
with respect to the unknown positive, bounded, and smooth
function $u(\text{x},t)$ on $\mathbb{R}^n\times(0,+\infty)$. Assume that the initial function $u_0(\text{x})$ satisfies the following conditions:
\begin{enumerate}
\item[$C_1)$]
$
u_0 \in \mathfrak{M} :=
\Bigl\{
\psi:\ \psi,\ \nabla_\text{x} \psi,\ \nabla_x\nabla_x \psi \in C_B(\mathbb{R}^n)
\Bigr\},
$\\
where
$
\nabla_\text{x} :=
\left\{
\frac{\partial}{\partial x_1},\ldots,
\frac{\partial}{\partial x_n}
\right\}.
$
\item[$C_2)$]
$
c_0:=\inf\limits_{\text{x}\in\mathbb{R}^n} u_0(\text{x}) >0.
$
\end{enumerate}

For $\lambda(\text{x},t)$ assume the following restrictions:

\begin{enumerate}
\item[$p_1)$]
$\lambda(\text{x},t)$ is continuously differentiable
on the set $\mathbb{R}^n\times(0,+\infty)$,
and both $\lambda$ and $\nabla_\text{x} \lambda$
are bounded on $\mathbb{R}^n\times(0,+\infty)$;

\item[$p_2)$]
there exist continuous functions
$\lambda_j(t)$, $j=1,2$,
on $[0,+\infty)$ with the following properties:
\begin{equation}\label{Khachatryan91}
0<\lambda_1(t)\le \lambda_2(t),\qquad t\in(0,+\infty),\qquad
\lambda_2\in L_1(0,+\infty),
\end{equation}
\begin{equation}\label{Khachatryan92}
\lambda_1(t)\not\equiv \lambda_2(t),\qquad t\in[0,+\infty),\qquad
0<\lim_{t\to 0+}\frac{\lambda_1(t)}{\lambda_2(t)}<+\infty,
\end{equation}
such that
\begin{equation}\label{Khachatryan93}
\lambda_1(t)\le \lambda(\text{x},t)\le \lambda_2(t),\qquad \text{x}\in\mathbb{R}^n,\ t>0.
\end{equation}
\end{enumerate}

Finally, the nonlinearity $G$ in equation \eqref{Khachatryan89} satisfies the following conditions:
\begin{enumerate}
\item[(G$_1$)] $G\in C[0,+\infty)$, $G(0)=0$, and the function $y=G(u)$ is monotonically increasing
on the set $[0,+\infty)$;
\item[(G$_2$)] the function $y=G(u)$ is strictly concave on the set $[0,+\infty)$;
\item[(G$_3$)] there exists a continuous, monotonically increasing, and strictly concave
mapping $\varphi:[0,1]\to[0,1]$ with properties $\varphi(0)=0$, $\varphi(1)=1$,
$\varphi'(+0)=+\infty$ such that
$$
G(\sigma u)\ge \varphi(\sigma)\,G(u),\qquad u\in\mathbb{R}_+,\ \sigma\in[0,1].
$$
\end{enumerate}

By a mild global positive solution of the problem \eqref{Khachatryan89}, \eqref{Khachatryan90}
we mean a continuous and positive on $\mathbb{R}^n\times(0,+\infty)$
solution of the following nonlinear integral equation (see \cite{korp16}):
\begin{equation}\label{Khachatryan94}
\begin{array}{c}
\displaystyle u(\text{x},t)=g(\text{x},t)+\int\limits_{0}^{t}\int\limits_{\mathbb{R}^n}
U(\text{x,y},t-s)\,\lambda(\text{y},s)\,G\bigl(u(\text{y},s)\bigr)\,d\text{y}\,ds,\\ \displaystyle\text{x}\in\mathbb{R}^n,\ t>0,
\end{array}
\end{equation}
where
\begin{equation}\label{Khachatryan95}
\begin{array}{c}
\displaystyle
g(\text{x},t):=\int\limits_{\mathbb{R}^n} U(\text{x,y},t)\,u_0(\text{y})\,d\text{y},
\\
\displaystyle U(\text{x,y},t)=\frac{1}{(4\pi t)^{\frac n2}}
\exp\!\left(-\frac{|\text{x}-\text{y}|^2}{4t}\right), \  \text{x,y}\in\mathbb{R}^n,\ t>0.
\end{array}
\end{equation}

A natural question arises: how are the Cauchy problem \eqref{Khachatryan89}, \eqref{Khachatryan90}
and the nonlinear integral equation \eqref{Khachatryan94} related, and when does a mild
global positive solution of the problem \eqref{Khachatryan89}, \eqref{Khachatryan90}
become smooth?

To answer this question, we present the following definitions
(see~\cite{korp16}, \cite{fuj17}).

Let $T>0$ be a given number. A nonnegative function $u(\text{x},t)$
will be called a classical solution of the problem
\eqref{Khachatryan89}, \eqref{Khachatryan90} on the interval $[0,T]$
if $u$, $\nabla_\text{x} u$, $\nabla_x\nabla_x u$, $\dfrac{\partial u}{\partial t}$
exist and are continuous in the strip
$
Q_T:=\mathbb{R}^n\times[0,T],
$
and for all $\text{x}\in\mathbb{R}^n$, $t\in(0,T]$ equation \eqref{Khachatryan89}
is satisfied, and
$
u(\text{x},0)=u_0(\text{x}), \ \text{x}\in\mathbb{R}^n.
$

Denote by $\varepsilon[0,T]$ the set of
functions $u(\text{x},t)$ continuous on $Q_T,$ for which there exist constants $M>0$
and $\gamma\in(0,2)$ such that
$
|u(\text{x},t)|\le M \exp\!\bigl(|\text{x}|^{\gamma}\bigr),
$ $ \text{x}\in\mathbb{R}^n,\ t\in(0,T].
$

Using arguments analogous to those in Lemma 8.4 of~\cite{korp16},
one can verify that if $u(\text{x},t)\geq0$ is a classical solution of the problem
\eqref{Khachatryan89}, \eqref{Khachatryan90} in $\varepsilon[0,T]$ for some $T>0$,
$u_0\in\mathfrak{M}$ and $u_0(\text{x})\ge 0$, $x\in\mathbb{R}^n$,
then $u(\text{x},t)$ satisfies the nonlinear integral equation
\eqref{Khachatryan94} on the set $\mathbb{R}^n\times(0,T]$.
Moreover, if the nonlinearity $G$ is continuously differentiable
on the set $(0,+\infty),$ and if $u(\text{x},t)$ is a nonnegative solution of the integral equation \eqref{Khachatryan94}, continuous
on $Q_T,$
and $u(\text{x},t)$ is bounded on $Q_T$,
then under condition $p_1)$ the function $u(\text{x},t)$
is a classical solution of the problem
\eqref{Khachatryan89}, \eqref{Khachatryan90}
(see Lemma 8.6 in~\cite{korp16}).

Consider the cone of nonnegative functions in the space
$C_B(\mathbb{R}^n\times(0,+\infty))$:
$
K:=\{\,v\in C_B(\mathbb{R}^n\times(0,+\infty)):\;
v(\text{x},t)\ge 0,\; \text{x}\in\mathbb{R}^n,\; t>0 \,\}.
$

Introduce the following nonlinear integral operator:
\begin{equation}\label{Khachatryan96}
\begin{array}{c}
\displaystyle(Av)(\text{x},t)
=
\int\limits_{0}^{t}\!
\int\limits_{\mathbb{R}^n}
U(\text{x},\text{y},t-s)\lambda(\text{y},s)
G\bigl(v(\text{y},s)+g(\text{y},s)\bigr)
\,d\text{y}\,ds,\\
\displaystyle \text{x}\in\mathbb{R}^n,\; t>0.
\end{array}
\end{equation}

Using conditions $C_1),C_2), G_1), p_1)$ and $p_2)$,
using \eqref{Khachatryan95} and the obvious equality
$
\int\limits_{\mathbb{R}^n} U(\text{x},\text{y},t)\,d\text{y} = 1,
\ \text{x}\in\mathbb{R}^n,\; t>0,
$
it is easy to verify that
$
A:K\to K,
$
and that the operator $A$ is monotone in $K.$ Now, using conditions $C_2),$ $p_2),$ $G_1),$ $G_3),$
using \eqref{Khachatryan95} from \eqref{Khachatryan96} we obtain
$$
(A\sigma v)(\text{x},t)
=
\int\limits_{0}^{t}
\int\limits_{\mathbb{R}^n}
U(\text{x},\text{y},t-s)\,\lambda(\text{y},s)\,
G\bigl(\sigma\,v(\text{y},s)+g(\text{y},s)\bigr)\,d\text{y}\,ds\ge
$$
$$\ge
\int\limits_{0}^{t}
\int\limits_{\mathbb{R}^n}
U(\text{x},\text{y},t-s)\,\lambda(\text{y},s)\,
G\bigl(\sigma\,(v(\text{y},s)+g(\text{y},s))\bigr)\,d\text{y}\,ds\ge
$$
$$\ge
\varphi(\sigma)
\int\limits_{0}^{t}
\int\limits_{\mathbb{R}^n}
U(\text{x},\text{y},t-s)\,\lambda(\text{y},s)\,
G\bigl(v(\text{y},s)+g(\text{y},s)\bigr)\,d\text{y}\,ds=
$$
$$=
\varphi(\sigma)\,(Av)(\text{x},t),
\qquad \text{x}\in\mathbb{R}^n,\ t>0,\ \sigma\in[0,1],\ v\in K.
$$
Finally, let us verify condition \eqref{Khachatryan28} for the operator $A$.
We choose $n_0$ to be $n_0=2,$
and $v_0$ to be any positive number $\xi$.
Then, from \eqref{Khachatryan96}, taking into account conditions $C_1)$, $C_2)$, $p_2)$, $G_1)$
and representation \eqref{Khachatryan95}, we obtain
$$
(A v_0)(\text{x},t)
=
\int\limits_{0}^{t}
\int\limits_{\mathbb{R}^n}
U(\text{x},\text{y},t-s)\,\lambda(\text{y},s)\,
G(\xi+g(\text{y},s))\,d\text{y}\,ds\le
$$
$$\le
G\!\left(\xi+\sup_{\text{x}\in\mathbb{R}^n}u_0(\text{x})\right)
\int\limits_{0}^{t}
\int\limits_{\mathbb{R}^n}
U(\text{x},\text{y},t-s)\,\lambda(\text{y},s)\,d\text{y}\,ds\le
$$
$$\le
G\!\left(\xi+\sup_{\text{x}\in\mathbb{R}^n}u_0(\text{x})\right)
\int\limits_{0}^{t}\lambda_2(s)\,ds,
\qquad \text{x}\in\mathbb{R}^n,\ t>0.
$$
On the other hand,
$$
(A v_0)(\text{x},t)
\ge
G(\xi+c_0)
\int\limits_{0}^{t}
\int\limits_{\mathbb{R}^n}
U(\text{x},\text{y},t-s)\,\lambda(\text{y},s)\,d\text{y}\,ds\ge
$$
$$\ge G(\xi+c_0)\int\limits_0^t\lambda_1(s)ds,\ \text{x}\in\mathbb{R}^n,\; t>0. $$
Thus, if we denote by $\beta_0:=\sup\limits_{\text{x}\in\mathbb{R}^n}u_0(\text{x})<+\infty,$ then, from the inequalities obtained above, we arrive at the following two–sided inequality:
\begin{equation}\label{Khachatryan97}
\begin{array}{c}
\displaystyle G(\xi+c_0)\int\limits_{0}^{t}\lambda_1(s)\,ds
\le
(Av_0)(\text{x},t)
\le
G(\xi+\beta_0)\int\limits_{0}^{t}\lambda_2(s)\,ds,
\\
\displaystyle \text{x}\in\mathbb{R}^n,\; t>0,
\end{array}
\end{equation}

Taking into account the monotonicity of the operator $A$ and using
conditions $C_1),C_2),p_2),G_1)$, as well as representation
\eqref{Khachatryan95} and the two–sided inequality \eqref{Khachatryan97}, from \eqref{Khachatryan96}
we obtain
$$
(A^2v_0)(\text{x},t)
\le$$$$\le
\int\limits_{0}^{t}\int\limits_{\mathbb{R}^n}
U(\text{x},\text{y},t-s)\lambda(\text{y},s)
G\!\left(
G(\xi+\beta_0)\int\limits_{0}^{s}\lambda_2(\tau)d\tau
+
g(\text{y},s)
\right)d\text{y}\,ds\le
$$
$$\le
G\!\left(
G(\xi+\beta_0)\|\lambda_2\|_{L_1(0,+\infty)}
+
\beta_0
\right)
\int\limits_{0}^{t}\lambda_2(s)\,ds ,
\qquad \text{x}\in\mathbb{R}^n,\; t>0 .
$$
$$
(A^2v_0)(\text{x},t)
\ge$$$$\ge
\int\limits_{0}^{t}\int\limits_{\mathbb{R}^n}
U(\text{x},\text{y},t-s)\lambda(\text{y},s)
G\!\left(
G(\xi+c_0)\int\limits_{0}^{s}\lambda_1(\tau)d\tau
+
g(\text{y},s)
\right)d\text{y}\,ds\ge
$$
$$\ge
G(c_0)\int\limits_{0}^{t}\lambda_1(s)\,ds ,
\qquad \text{x}\in\mathbb{R}^n,\; t>0 .
$$

Thus we arrive at the following chain of inequalities:
\begin{equation}\label{Khachatryan98}
\begin{array}{c}
\displaystyle G(c_0)\int\limits_{0}^{t}\lambda_1(s)\,ds
\le
(A^2v_0)(\text{x},t)
\le\\
\displaystyle\le
G\!\left(
G(\xi+\beta_0)\|\lambda_2\|_{L_1(0,+\infty)}
+
\beta_0
\right)
\int\limits_{0}^{t}\lambda_2(s)\,ds, \text{x}\in\mathbb{R}^n, t>0.
\end{array}
\end{equation}

From \eqref{Khachatryan97} and \eqref{Khachatryan98} it follows that
\begin{equation}\label{Khachatryan99}
\begin{array}{c}
\displaystyle\frac{G(c_0)\displaystyle\int\limits_0^t \lambda_1(s)\,ds}
     {G(\xi+\beta_0)\displaystyle\int\limits_0^t \lambda_2(s)\,ds}
\,(Av_0)(\text{x},t)
\le
(A^2v_0)(\text{x},t)\le\\
\displaystyle\le
\frac{G\!\left(G(\xi+\beta_0)\|\lambda_2\|_{L_1(0,+\infty)}+\beta_0\right)
      \displaystyle\int\limits_0^t \lambda_2(s)\,ds}
     {G(\xi+c_0)\displaystyle\int\limits_0^t \lambda_1(s)\,ds}
\,(Av_0)(\text{x},t),\\
 \text{x}\in\mathbb{R}^n,\ t>0.
\end{array}
\end{equation}

Consider the function
\[
\chi(t):=
\frac{\displaystyle\int\limits_0^t \lambda_1(s)\,ds}
     {\displaystyle\int\limits_0^t \lambda_2(s)\,ds},
\qquad t>0.
\]

From condition $p_2)$ it immediately follows that
\[
\chi\in C(0,+\infty),\qquad
\chi(t)\le1,\quad t>0,\qquad
\chi(+\infty)=
\frac{\displaystyle\int\limits_0^\infty \lambda_1(s)\,ds}
     {\displaystyle\int\limits_0^\infty \lambda_2(s)\,ds}
\le1,
\]
\[
\lim_{t\to0+}\chi(t)
=
\lim_{t\to0+}\frac{\lambda_1(t)}{\lambda_2(t)}
<+\infty,
\qquad
\lim_{t\to0+}\chi(t)>0.
\]

Consequently,
$
\chi\in C[0,+\infty)
,$
and there exists a number $\delta_0\in(0,1)$ such that
\begin{equation}\label{Khachatryan100}
\chi(t)\ge\delta_0,\qquad t\in\mathbb{R}_+ .
\end{equation}

Thus, using \eqref{Khachatryan100} and the monotonicity of $G,$ we obtain
\[
r_1(Av_0)(\text{x},t)\le (A^2v_0)(\text{x},t)\le r_2(Av_0)(\text{x},t),
\qquad \text{x}\in\mathbb{R}^n,\ t>0,
\]
where
\[
r_1:=\delta_0\,\frac{G(c_0)}{G(\xi+\beta_0)}\in(0,1),
\]
\[
r_2:=\frac1{\delta_0}
\max\!\left\{
1,
\frac{G\!\left(G(\xi+\beta_0)\|\lambda_2\|_{L_1(0,+\infty)}+\beta_0\right)}
     {G(\xi+c_0)}
\right\}\in(1,+\infty).
\]
Hence, for operator $A$ condition \eqref{Khachatryan28}  is satisfied.
From \eqref{Khachatryan95} and conditions $C_1)$, $p_1)$, $G_1)$ it follows that the operator $A$
is continuous in the cone $K$.
Therefore, taking into account the normality of the cone $K$ and using Theorem~\ref{thm3},
we conclude that the operator $A$ has a nontrivial
fixed point $v_*(\text{x},t)$  in the cone $K.$ Moreover
\begin{equation}\label{Khachatryan101}
\tau_1(Av_0)(\text{x},t)\le v_*(\text{x},t)\le \delta_1(Av_0)(\text{x},t),
\qquad \text{x}\in\mathbb{R}^n,\; t>0,
\end{equation}
where $\tau_1\in(0,r_1)$, $\delta_1\in(r_2,+\infty)$ are
the unique solutions of the characteristic equations
$
\varphi(\tau)=\frac{\tau}{r_1},
\
\delta\,\varphi\!\left(\frac1\delta\right)=r_2,
$
respectively.
Moreover, statement 2) of Theorem~\ref{thm3} holds
with respect to the norm of the space $C_B(\mathbb{R}^n\times(0,+\infty))$.

Now observe that the function
$
u_*(\text{x},t):=v_*(\text{x},t)+g(\text{x},t)
$
is a continuous, positive, and bounded on
$\mathbb{R}^n\times(0,+\infty)$ solution of the nonlinear integral
equation \eqref{Khachatryan94}.
Using the two–sided estimate \eqref{Khachatryan101} and Theorem~\ref{thm7},
it is easy to verify that, apart from solution $u_*(\text{x},t),$ equation \eqref{Khachatryan94}
 has no other solutions in the cone $K.$

It should also be noted that the methods we developed for constructing
fixed points of monotone and strongly concave operators
can also be successfully applied to the Cauchy problem for semilinear
wave equations.\\

\section*{Acknowledgements}

The research of the author was carried out with the financial support of the
Science Committee of the Republic of Armenia within the framework of
scientific project No.~23RL--1A027.\\


\bigskip
\noindent
\textbf{Khachatur Aghavardovich Khachatryan}

Dr. Phys.-Math. Sci., Professor

Department of Mathematics and Mechanics

Yerevan State University, 0025, Alex Manoogyan Str.1

E-mail: \texttt{khachatur.khachatryan@ysu.am}


\begin{thebibliography}{99}

\bibitem{kras1}
Krasnosel'skii, M.A.:
Positive solutions of operator equations.
P. Noordhoff, USA  (1964)

\bibitem{khach2}
Khachatryan, Kh.A.:
On a class of nonlinear integral equations with a noncompact operator.
J. Contemp. Math. Anal.
\textbf{46}(2),  89--100 (2011)

\bibitem{bakh3}
 Bakhtin, I.A., Krasnosel'skii, M.A.:
The method of successive approximations in the theory of equations with concave operators.
Sibirsk. Mat. Zh. \textbf{2}(3), 313--330 (1961)  (in Russian)

\bibitem{kras4}
 Krasnosel'skii, M.A., Stetsenko, V.Ya.:
Toward a theory of equations with concave operators.
Sib. Math. J.,
\text{10}(3), 405--410 (1969).

\bibitem{bre5}
Brekke, L., Freund,  P.G.O.,  Olson, M.,   Witten, E.:
Non-archimedean string dynamics.
Nuclear Physics B,
\textbf{302}(3),  365--402 (1988)

\bibitem{vla6}
Vladimirov, V.S., Volovich,  Ya.I.:
Nonlinear dynamics equation in $p$-adic string theory.
Theor. Math. Phys. \textbf{138}(3),  297--309 (2004)

\bibitem{diek7}
 Diekmann, O.: Thresholds and travelling waves for the geographical spread of infection.
J. Math. Biol. \textbf{6}(2),  109--130 (1978)

\bibitem{diek8}
Diekmann, O.,  Kaper, H.G.: On the bounded solutions of a nonlinear convolution equation.
Nonlinear Analysis \textbf{2}(6), 721--734 (1978)

\bibitem{khach9}
Khachatryan, A.Kh., Khachatryan, Kh.A.,  Petrosyan, H.S.: On the constructive solvability of one class nonlinear integral equations of the Hammerstein type on the whole line. Russian Math. (Iz. VUZ) \textbf{69}(3),  77--93 (2025)

\bibitem{aref10}
Aref'eva, I.Ya., Koshelev, A.S., Joukovskaya, L.V.:
Time evolution in superstring field theory on non-BPS brane.
I. Rolling tachyon and energy--momentum conservation.
J. High Energy Phys. \textbf{2003}(9), 1--15 (2003)

\bibitem{gel11}
Gel'fand, I.M.,  Shilov, G.E.: Spaces of fundamental and generalized functions.
Amer. Math. Soc.  (2016)

\bibitem{Io}
Iosevich, A., Liflyand, E.: Decay of the Fourier Transform.  Analytic and Geometric Aspects. Birkhäuser Basel, Springer Basel (2014)

\bibitem{KV}
Karapetyants, A.N., Vagharshakyan, A.A.:
The Hadamard-Bergman convolution on the half-plane.
J. Fourier Anal. Appl. \textbf{30}, no. 38 (2024)

\bibitem{khach12}
Khachatryan, Kh.A.: On the solvability of a boundary value problem in $p$-adic string theory. Trans. Moscow Math. Soc.\textbf{79}(1),  101--115 (2018)

\bibitem{eng13}
Engibaryan, N.B.: On a problem in nonlinear radiative transfer. Astrophysics  \textbf{2}(1), 12--14 (1966)

\bibitem{cerc14}
Cercignani, C.:
The Boltzmann equation and its applications.
Appl. Math. Sci., Springer-Verlag,
New York (1988)

\bibitem{khach15}
Khachatryan, A.Kh., Khachatryan, Kh.A.:
A one-parameter family of positive solutions of the nonlinear stationary Boltzmann equation (in the framework of a modified model).
Russian Math. Surveys  \textbf{72}(3), 571--573 (2017)

\bibitem{korp16}
Korpusov, M.O., Sveshnikov, A.G.:
Nonlinear Functional Analysis and Mathematical Modeling in Physics.
Methods for Studying Nonlinear Operators.
URSS, Moscow, (2011) (in Russian)

\bibitem{fuj17}
Fujita, H.: On the blowing up of solutions of the Cauchy problem for $u_t=\Delta u+u^{1+\alpha}.$ J. Fac. Sci. Univ. Tokyo Sect. I \textbf{13}, 109--124 (1966)

\end{thebibliography}
\end{document}